\documentclass{amsart}

\usepackage{amsmath}
\usepackage{amssymb}
\usepackage{amsthm}

\newcommand{\theoremref}[1]{\hyperref[#1]{Theorem~\ref*{#1}}}
\newcommand{\claimref}[1]{\hyperref[#1]{Claim~\ref*{#1}}}
\newcommand{\situationref}[1]{\hyperref[#1]{Situation~\ref*{#1}}}
\newcommand{\lemmaref}[1]{\hyperref[#1]{Lemma~\ref*{#1}}}
\newcommand{\definitionref}[1]{\hyperref[#1]{Definition~\ref*{#1}}}
\newcommand{\propositionref}[1]{\hyperref[#1]{Proposition~\ref*{#1}}}
\newcommand{\conjectureref}[1]{\hyperref[#1]{Conjecture~\ref*{#1}}}
\newcommand{\corollaryref}[1]{\hyperref[#1]{Corollary~\ref*{#1}}}
\newcommand{\exerciseref}[1]{\hyperref[#1]{Exercise~\ref*{#1}}}

\usepackage[mathscr]{euscript}

\usepackage[all]{xy}

\xyoption{2cell}
\UseAllTwocells

\usepackage{multicol}

\usepackage{xr-hyper}

\usepackage{hyperref}
\usepackage[usenames,dvipsnames]{xcolor}
\hypersetup{colorlinks=true,citecolor=NavyBlue,linkcolor=BrickRed,urlcolor=Orange}

\numberwithin{equation}{section}

\theoremstyle{plain}
\newtheorem{theorem}[equation]{Theorem}
\newtheorem{proposition}[equation]{Proposition}
\newtheorem{lemma}[equation]{Lemma}

\theoremstyle{definition}
\newtheorem{definition}[equation]{Definition}

\theoremstyle{remark}
\newtheorem{remark}[equation]{Remark}

\newtheorem{notation}[equation]{Notation}

\usepackage{tikz}
\usepackage{tikz-cd}
\usetikzlibrary{matrix,arrows}
\usetikzlibrary{decorations.pathmorphing}

\renewcommand{\P}{\mathbb{P}}

\DeclareMathOperator{\Pic}{Pic}

\DeclareMathOperator{\Proj}{Proj}

\usepackage{tikz}
\usepackage{amscd}

\newcommand{\git}{\mathbin{
  \mathchoice{/\mkern-6mu/}
    {/\mkern-6mu/}
    {/\mkern-5mu/}
    {/\mkern-5mu/}}}

\def\Pic{\operatorname{Pic}}

\newcommand{\bfP}{\mathbf{P}}
\newcommand{\bP}{\mathbb{P}}

\newcommand{\bZ}{\mathbb{Z}}
\newcommand{\bQ}{\mathbb{Q}}

\newcommand{\bC}{\mathbb{C}}

\newcommand{\Aut}{\mathrm{Aut}}

\newcommand{\Res}{\mathrm{Res}}

\newcommand{\spec}{\mathrm{Spec}\;} 
\newcommand{\proj}{\mathrm{Proj}\;} 
\newcommand{\Bl}{\mathrm{Bl}}

\newcommand{\SL}{\mathrm{SL}}

\newcommand{\calD}{\mathcal{D}} 
\newcommand{\calH}{\mathcal{H}} 

\newcommand{\calP}{\mathcal{P}}

\newcommand{\calO}{\mathcal{O}}
\newcommand{\calL}{\mathcal{L}}

\newcommand{\calC}{\mathcal{C}}

\newcommand{\calM}{\mathcal{M}}
\newcommand{\calY}{\mathcal{Y}}

\newcommand{\calU}{\mathcal{U}}

\begin{document}

\title[]{On the moduli space of pairs consisting of a cubic threefold and a hyperplane}
\author[]{Radu Laza} 
\address{Mathematics Department, Stony Brook University, Stony Brook, NY 11794-3651}
\email{radu.laza@stonybrook.edu}
\author[]{Gregory Pearlstein}
\address{Department of Mathematics, Texas A\&M University, College Station, TX 77843-3368}
\email{gpearl@math.tamu.edu}
\author[]{Zheng Zhang}
\address{Department of Mathematics, University of Colorado, Boulder, CO 80309-0395}
\email{zheng.zhang-2@coloradu.edu}
\date{\today}
\thanks{The authors acknowledge partial support from NSF Grants DMS-1254812, DMS-1361143 (Laza) and DMS-1361120 (Pearlstein and Zhang).}

\bibliographystyle{amsalpha}
\maketitle

\begin{abstract}
We study the moduli space of pairs $(X,H)$ consisting of a cubic threefold $X$ and a hyperplane $H$ in $\bP^4$. The interest in this moduli comes from two sources: the study of certain weighted hypersurfaces whose middle cohomology admit Hodge structures of $K3$ type and, on the other hand, the study of the singularity $O_{16}$ (the cone over a cubic surface). In this paper, we give a Hodge theoretic construction of the moduli space of cubic pairs by relating $(X,H)$ to certain ``lattice polarized" cubic fourfolds $Y$. A period map for the pairs $(X,H)$ is then defined using the periods of the cubic fourfolds $Y$. The main result is that the period map induces an isomorphism between a GIT model for the pairs $(X,H)$ and the Baily-Borel compactification of some locally symmetric domain of type IV. 
\end{abstract}

\section*{Introduction}
It is an interesting problem to study (weighted) hypersurfaces whose middle cohomology admit Hodge structures of $K3$ type (a Hodge structure is of $K3$ type if it is an effective weight $2$ Hodge structure with $h^{2,0}=1$) up to Tate twist. In 1979, Reid listed all $95$ families of $K3$ weighted hypersurfaces (see for example \cite{Reid_weightedk3}). The next case is to consider quasi-smooth hypersurfaces of degree $d$ in a weighted projective space $\bP(w_0,w_1,w_2,w_3,w_4,w_5)$. By Griffiths residue calculus, if $d = \frac12(w_0+\dots+w_5)$ then the middle cohomology of the weighted fourfolds are of $K3$ type (see Appendix \ref{wps} for more discussions on $K3$ type varieties). In particular, when $w_0 = \dots = w_5 =1$ one obtains cubic fourfolds. These weighted Fano hypersurfaces are of interest to us because they might be used to construct hyper-K\"ahler manifolds. 
For instance, cubic fourfolds have been a rich source for producing hyper-K\"ahler manifolds (e.g. \cite{BD_fano}, \cite{LLSvS_twistcubic} and \cite{LSV_hyperkahler}). 

Families of $K3$ type weighted fourfolds which contain a Fermat member can be easily classified (c.f. \theoremref{quasi-k3-4fold} and Table \ref{wps-table}). There are $17$ cases most of which have essentially appeared in Reid's list. In this paper we consider one of the new cases, namely, weighted degree $6$ hypersurfaces in $\bP(1,2,2,2,2,3)$. A quasi-smooth hypersurface $Z$ of degree $6$ in $\bP(1,2,2,2,2,3)$ is a double cover of $\bP^4$ branched along a smooth cubic threefold $X$ and a hyperplane $H$. The isomorphism class of $Z$ is determined by the projective equivalence class of the branching data $X$ and $H$. Thus, we are interested in the moduli space of pairs $(X, H)$ where $X$ is a cubic threefold and $H$ is a hyperplane in $\bP^4$. This moduli space is also interesting from the perspective of singularity theory. By the theory of Pinkham \cite{Pinkham_Gm} the study of deformations of the singularity $O_{16}$ (which comes from the affine cone over a cubic surface) is essentially reduced to the study of the moduli of cubic pairs $(X,H)$. The strategy has been successfully applied to unimodal singularities by Pinkham, Looijenga and Brieskorn, and to the minimally elliptic surface singularity $N_{16}$ by the first author (see \cite{Laza_n16} and references therein). The singularity $O_{16}$ is the simplest genuinely threefold singularity (i.e. not obtained by suspending a surface singularity). The understanding of the deformations of $O_{16}$ is essential to any attempt of studying deformations of higher modality singularities. Note also that there is a close relation between the deformations of $N_{16}$ and $O_{16}$. We will discuss more about $O_{16}$ elsewhere. 

We study the moduli of pairs $(X,H)$ consisting of a cubic threefold $X$ and a hyperplane $H$ via a period map. To construct periods for pairs $(X,H)$, we use a variation of the construction by Allcock, Carlson and Toledo \cite{ACT_cubicsurf, ACT_cubic3} which allows us to encode a pair $(X,H)$ as a cubic fourfold $Y$. A cubic fourfold $Y$ coming from a cubic pair $(X,H)$ is characterized by the geometric property that it admits an involution fixing a hyperplane section, or equivalently, it has an Eckardt point (i.e. $Y$ contains a cone over a cubic surface as a hyperplane section and we call the vertex an Eckardt point). Note that our construction works for cubic pairs of any dimension. A smooth cubic fourfold $Y$ admitting an Eckardt point contains at least $27$ planes (generated by the $27$ lines on the cubic surface and the Eckardt point). The Hodge classes corresponding to these planes generate a saturated sublattice $M \subset H^4(Y, \bZ) \cap H^{2,2}(Y)$. In addition to the hyperplane class, the lattice $M$ contains a (scaled) $E_6$ lattice induced from the primitive cohomology of the cubic surface $X \cap H$. The cubic fourfolds $Y$ are characterized Hodge theoretically as ``$M$-polarized" cubic fourfolds (analogous to lattice polarized $K3$ surfaces defined by Dolgachev \cite{Dolgachev_latticek3}). 

\begin{remark}
After the completion of this work, we have learned that $M$-polarized cubic fourfolds as above (or equivalently cubics with an Eckardt point) are an interesting testing ground for various natural rationality conjectures for cubic fourfolds (see \cite{Laza_maxirrat}).
\end{remark}

Using periods of the cubic fourfolds $Y$ we get a period map $\calP_0$ for cubic pairs $(X,H)$. It is well known that the periods of cubic fourfolds behave similarly to the periods of $K3$ surfaces. The Torelli theorem for cubic fourfolds was proven by Voisin \cite{Voisin_cubic}, ``lattice polarized" cubic fourfolds were used by Hassett \cite{Hassett_cubic}, and the image of the period map was described by Looijenga \cite{Looijenga_cubic} and the first author \cite{Laza_cubicgit, Laza_cubic}. Using Voisin's Torelli theorem and lattice theory, we first prove that the moduli space of cubic pairs $(X,H)$ is birational to a certain locally symmetric domain of type IV. More precisely, we let $T=M_{H^4(Y, \bZ)}^{\perp}$. Denote a connected component of the period domain for weight $4$ Hodge structures on $T$ with Hodge numbers $[0,1,14,1,0]$ by $\calD_M$. Also set $O^+(T) \subset O(T)$ to be the group of isometries of $T$ stabilizing $\calD_M$.

\begin{theorem}[= \theoremref{P0 injective}] \label{mainthm1}
Let $\calM_0$ be the moduli of pairs $(X,H)$ consisting of a cubic threefold $X$ and a hyperplane $H$ in $\bP^4$ such that the associated cubic fourfold $Y$ is smooth. The period map $\calP_0: \calM_0 \rightarrow \calD_M/O^+(T)$ defined via periods of the cubic fourfolds $Y$ is an isomorphism onto the image.
\end{theorem}

Next we consider the problem of compactifying the period map $\calP_0$. There is a natural compactification for the moduli of cubic pairs $(X,H)$ using geometric invariant theory (GIT). Specifically, we take the GIT quotient of $ \bP H^0(\bP^4, \calO(3)) \times \bP H^0(\bP^4, \calO(1))$ with respect to the action of $\SL(5, \bC)$. Note that our GIT construction depends on a parameter $t$ which corresponds to the choice of a linearization (see for example \cite{Laza_n16}). Write $\calM(t):=\bP H^0(\bP^4, \calO(3)) \times \bP H^0(\bP^4, \calO(1)) \git_t \SL(5,\bC)$. The natural question for us is how does $\calM(t)$ compare to the Baily-Borel compactification of $\calD_M/O^+(T)$ (for some related examples see \cite{Shah_deg2k3}, \cite{LS_cubic3fold}, \cite{Looijenga_cubic}, \cite{Laza_cubic} and \cite{ACT_cubic3}). 
Note that $\calM(t)$ is not empty when $0 \leq t \leq \frac34$ (as $t$ increases, $X$ becomes more singular while transversality for $X \cap H$ becomes better). Also, $\calM(t)$ and $\calM(t')$ only differ in codimension $2$ for $t \neq t'$. Our second result is that the GIT moduli $\calM(\frac13)$ at $t=\frac13$ is isomorphic to the Baily-Borel compactification of $\calD_M/O^+(T)$. 

\begin{theorem}[= \theoremref{gitbb}] \label{mainthm2}
The period map $\calP_0: \calM_0 \rightarrow \calD_M/O^+(T)$ extends to an isomorphism $\calM(\frac13) \cong (\calD_M/O^+(T))^*$ where $(\calD_M/O^+(T))^*$ denotes the Baily-Borel compactification of $\calD_M/O^+(T)$. 
\end{theorem}

To prove this theorem, we apply a general framework developed by Looijenga \cite{Looijenga_ball, Looijenga_typeiv} (see also \cite{Looijenga_vancouver} and \cite{LS_openperiod}) for comparing GIT compactifications to appropriate compactifications of period spaces. Specifically, we consider cubic pairs $(X,H)$ with $X$ at worst nodal and $H$ at worst simply tangent to $X$. Such pairs are stable for any $0 < t < \frac34$. Denote the corresponding moduli by $\calM$ (N.B. $\calM_0 \subset \calM$ and the complement of $\calM$ in $\calM(t)$ has codimension at least $2$). The period map $\calP_0$ extends to $\calP: \calM \rightarrow \calD_M/O^+(T)$. Using the work of Looijenga \cite{Looijenga_cubic} and the first author \cite{Laza_cubicgit, Laza_cubic}, we are able to show that the complement of the image $\calP(\calM)$ in $\calD_M/O^+(T)$ has codimension at least $2$. Indeed, the image of the period map for smooth cubic fourfolds is the complement of a union of the hyperplane arrangements $\calH_{\Delta}$ (which corresponds to the limiting mixed Hodge structures of cubic fourfolds with simple (A-D-E) singularities, see \cite[\S 4.2]{Hassett_cubic}) and $\calH_{\infty}$ (which parameterizes the limiting mixed Hodge structures of certain determinantal cubic fourfolds, see \cite[\S 4.4]{Hassett_cubic}). The hyperplanes in $\calH_{\infty}$ do not intersect with the subdomain $\calD_M$. The intersection of the hyperplane arrangement $\calH_{\Delta}$ with $\calD_M$ gives two Heegner divisors $H_n$ (the nodal Heegner divisor) and $H_t$ (the tangential Heegner divisor) in $\calD_M/O^+(T)$. Geometrically, they correspond respectively to the divisors in $\calM$ parameterizing pairs $(X,H)$ with $X$ nodal or $H$ simply tangent to $X$. To complete the proof, we need to select  $t \in (0,\frac34)$ such that the natural polarization on the GIT quotient $\calM(t)$ matches with the polarization on the Baily-Borel compactification $(\calD_M/O^+(T))^*$ (recall that $\calM$ is an open subset with boundary of codimension $\ge 2$ in both). The polarization on the period domain side is given by the Hodge bundle $\lambda(O^+(T))$. Since the Picard rank of $\calD_M/O^+(T)$ is $2$, all we need to do is to express $\lambda(O^+(T))$ as a linear combination of the Heegner divisors $H_n$ and $H_t$ (and then pull it back to the moduli side). This is by now standard, we call it a Borcherds' relation, and follows by similar computations to those in \cite{LOG_quartick3} (see also \cite{Kondo_kodaira2}, \cite{GHS_kodaira} and \cite{TV-A_kodaira}).  

\begin{remark}
Consider $2$-dimensional cubic pairs $(S,D)$ consisting of a cubic surface $S$ and a hyperplane $D$. Following \cite{ACT_cubicsurf} we associate a cubic threefold $X$ (which is the triple cyclic cover of $\bP^3$ branched along $S$ and hence admits an automorphism of order $3$) to $S$. We can also view $D$ as a hyperplane in $\bP^4$. Our construction in this paper for $(X, D)$ seems to give a complex ball uniformization for the moduli space of $(S,D)$. Specifically, the associated cubic fourfold $Y$ is not only ``lattice polarized" but also admits an automorphism of order $3$; the corresponding Mumford-Tate subdomain is a $7$-dimensional ball. 
\end{remark}

\begin{remark}
We point out that after the appearance of our manuscript, Chenglong Yu and Zhiwei Zheng \cite{YZ_symcubic} have made a systematic study of the moduli space of symmetric cubic fourfolds (i.e. cubics with specified automorphism group). Since cubics with Eckardt points can be characterized as cubics admitting an involution that fixes a hyperplane, they fit into the loc. cit. framework. While some overlap between our results and those of  Yu--Zheng exist (see esp.  \cite[Prop. 6.5]{YZ_symcubic}), the focus of the two papers is essentially complementary. 
\end{remark}

\subsection*{Acknowledgement} The first author wishes to thank I. Dolgachev, R. Friedman, E. Looijenga and K. O'Grady for some discussions on topics related to this paper. The second author and the third author are grateful to E. Izadi for many valuable suggestions. We would also like to thank P. Gallardo, B. Hassett, Zhiyuan Li, J. Martinez-Garcia, Zhiyu Tian and A. V\'arilly-Alvarado for useful discussions. 

\section{Associate a cubic fourfold to a cubic threefold and a hyperplane} \label{associated cubic4}
In this section we construct a smooth cubic fourfold $Y$ from a pair $(X,H)$ consisting of a smooth cubic threefold $X$ and a hyperplane $H$ intersecting $X$ transversely. Moreover, we characterize the cubic fourfold $Y$ geometrically by the fact that it contains a cone over a smooth cubic surface (more precisely, $Y$ contains an Eckardt point) and admits a certain involution. The construction is inspired by the work of Allcock, Carlson and Toledo \cite{ACT_cubic3}. Given a smooth cubic threefold in $\bP^3$ they consider the triple cyclic cover of $\bP^4$ branched along the cubic threefold. In our case we take the double cover $Z$ of $\bP^4$ branched over $X + H$. Note that $Z$ is isomorphic to a quasi-smooth hypersurface of degree $6$ in $\bP(1,2,2,2,2,3)$. The cubic fourfold $Y$ is obtained as a birational modification of the double cover $Z$ (blow up the ramification locus over $X \cap H$ and then blow down the strict transform of the preimage of $H$ in $Z$). Because our construction works for cubic pairs of any dimension, we shall present it in that generality. 

Let $X \subset \bP^{n}$ ($n \geq 3$) be a smooth cubic $(n-1)$-fold cut out by a cubic homogeneous polynomial $f(y_0, \dots, y_n) = 0$. Let $H$ be a hyperplane intersecting $X$ transversely. Assume the equation of $H$ is $l(y_0, \dots, y_n) = 0$. Note that the intersection $X \cap H$ is a smooth cubic $(n-2)$-fold in $H \cong \bP^{n-1}$. We consider the cubic $n$-fold $Y \subset \bP^{n+1}$ defined by 
\begin{equation} \label{Y eqn}
f(y_0, \dots, y_n) + l(y_0, \dots, y_n)y_{n+1}^2 = 0
\end{equation}
where $y_{n+1}$ is a new variable. (In what follows we view $X$ and $H$ as subvarieties of the hyperplane $(y_{n+1}=0) \cong \bP^n$.)

\begin{remark}
Geometrically, the cubic $n$-fold $Y$ can be constructed as follows (see \propositionref{Y projection}). Take the double cover $Z$ of $\bP^n$ ramified over the union $X+H$. After blowing up $Z$ along the reduced inverse image of $X \cap H$ and then blowing down the strict transform of the inverse image of $H$ in $Z$, we obtain $Y$.
\end{remark}

\begin{lemma} \label{Y smooth}
Notations as above. The cubic $n$-fold $Y$ defined by $f(y_0, \dots, y_n) + l(y_0, \dots, y_n)y_{n+1}^2 = 0$ is smooth.
\end{lemma}
\begin{proof}
Set $F(y_0, \dots, y_{n+1}) = f(y_0, \dots, y_n) + l(y_0, \dots, y_n)y_{n+1}^2$. A singular point is a common zero point for both $F$ and the vector $(\frac{\partial F}{\partial y_0}, \dots, \frac{\partial F}{\partial y_i}, \dots, \frac{\partial F}{\partial y_n}, \frac{\partial F}{\partial y_{n+1}})$ which is $(\frac{\partial f}{\partial y_0} + \frac{\partial l}{\partial y_0}y_{n+1}^2, \dots, \frac{\partial f}{\partial y_i} + \frac{\partial l}{\partial y_i}y_{n+1}^2, \dots, \frac{\partial f}{\partial y_n} + \frac{\partial l}{\partial y_n}y_{n+1}^2, 2ly_{n+1})$. Suppose $ly_{n+1}=0$. There are two possibilities: either $y_{n+1} = 0$ or $l(y_0, \dots, y_n)=0$. If $y_{n+1}=0$ (and hence $f(y_0, \dots, y_n)=0$) then the vector must be nonzero because $X$ is smooth. Now suppose $l(y_0, \dots, y_n) = 0$ (which implies that $f(y_0, \dots, y_n)=0$). Notice that $[0,\dots,0,1]$ is a smooth point of $Y$. Since $H$ intersects $X$ transversely, the following matrix has rank $2$ for every point on $X \cap H$ (in particular, $X \cap H$ is smooth): 
\[
\begin{pmatrix}
    \frac{\partial f}{\partial y_0}  & \dots & \frac{\partial f}{\partial y_i} & \dots & \frac{\partial f}{\partial y_n} \\
    \frac{\partial l}{\partial y_0}  & \dots & \frac{\partial l}{\partial y_i} & \dots & \frac{\partial l}{\partial y_n} \\ 
   \end{pmatrix}.
\]
As a result, the vector $(\frac{\partial F}{\partial y_0}, \dots, \frac{\partial F}{\partial y_n}, \frac{\partial F}{\partial y_{n+1}})$ is nonzero. 
\end{proof}

In a similar way one can prove the converse of \lemmaref{Y smooth}.
\begin{lemma} \label{(X,H) smooth}
Consider a cubic $n$-fold $Y \subset \bP^{n+1}$ with equation $$f(y_0, \dots, y_n) + l(y_0, \dots, y_n)y_{n+1}^2 = 0.$$ Let $X$ (resp. $H$) be the variety cut out by the equations $f(y_0, \dots, y_n)=y_{n+1}=0$ (resp. $l(y_0, \dots, y_n)=y_{n+1}=0$). If $Y$ is smooth, then $X$ is smooth. Moreover, the linear subspace $H$ intersects $X$ transversely. 
\end{lemma}

Now we study the geometry of the smooth cubic $n$-fold $Y$. From the equation it is easy to see that $Y$ contains the point $p=[0,\dots,0,1]$. Furthermore, the hyperplane section of $Y$ defined by $l(y_0, \dots, y_n)=0$ is a cone over the smooth cubic $(n-2)$- fold $X \cap H$ with vertex $p$. As a generalization of an Eckardt point for a cubic surface (see for example \cite[Chap. 9]{Dolgachev_cag}), we call $p$ an Eckardt point of $Y$. The cubic $n$-fold $Y$ coming from a cubic $(n-1)$-fold $X$ and a hyperplane $H$ is characterized geometrically by the fact that it contains an Eckardt point.

\begin{definition} 
A point $p$ on a cubic hypersurface $V \subset \mathbb{P}^{n+1}$ is called an \emph{Eckardt point} if the following two conditions hold:
\begin{enumerate}
\item $p$ is a smooth point of $V$;
\item $p$ is a point of multiplicity $3$ for the $(n-1)$-dimensional cubic $V \cap T_pV$ (where $T_pY$ denotes the projective tangent space at $p \in Y$).  
\end{enumerate} 
\end{definition}

Eckardt points are also called star points by other authors (cf. \cite{CC_star}). If $V$ is smooth, then the second condition is equivalent to saying that $(V \cap T_pV) \subset T_pY$ is a cone with vertex $p$ over a $(n-2)$-dimensional cubic hypersurface. It is also easy to prove the following lemma.

\begin{lemma} \label{Eckardt eqn}
If a cubic $n$-fold $V$ contains an Eckardt point, then we can choose coordinates such that the equation defining $V$ is $G(y_0, \dots,y_n) = g(y_0,\dots,y_n)+y_0y_{n+1}^2$ where $g$ is a homogeneous cubic polynomial.
\end{lemma}

We give a second geometric characterization of $Y$. There is a natural involution
\begin{equation} \label{sigma}
\sigma: [y_0, \dots, y_n, y_{n+1}] \mapsto [y_0, \dots, y_n, -y_{n+1}]
\end{equation} 
acting on the smooth cubic $n$-fold $Y=(f(y_0, \dots, y_n) + l(y_0, \dots, y_n)y_{n+1}^2 = 0)$. Clearly $\sigma$ fixes the hyperplane $(y_{n+1}=0)$ pointwise and also the point $p=[0,\dots,0,1]$. Conversely, we have the following lemma. Note that the fixed locus of an involution of $\bP^{n+1}$ is the union of two subspaces of dimensions $k$ and $n-k$.

\begin{lemma}\label{leminvo}
Let $V \subset \bP^{n+1}$ be a smooth cubic $n$-fold. Suppose $\tau$ is an involution of $\bP^{n+1}$ preserving $V$ and the fixed locus of $\tau$ consists of a hyperplane and a point $p$ which belongs to $V$. Then there exist coordinates such that $V$ is cut out by the equation $f(y_0,\dots,y_n)+l(y_0,\dots,y_n)y_{n+1}^2=0$ where $f$ (resp. $l$) is a cubic (resp. linear) homogeneous polynomial. 
\end{lemma}
\begin{proof}
Suppose $V$ admits a projective isomorphism of order $2$ with one isolated fixed point $p$ on $V$. Choose coordinates on $\P^5$ such that $\tau$ is given by $[y_0, \dots, y_n, y_{n+1}] \mapsto [y_0, \dots, y_n, -y_{n+1}]$. The fixed point is $p=[0,\dots,0,1]$. Then $V$ has equation $y_{n+1}^2l(y_0,\dots,y_n)+y_{n+1} q(y_1,\dots,y_n)+f(y_0,\dots,y_n)=0$. Because $\tau$ preserves $V$ we have $q=0$.   
\end{proof}

Now let us consider the projection of the smooth cubic $n$-fold $Y$ from the Eckardt point $p=[0, \dots, 0, 1]$ to the hyperplane $(y_{n+1}=0)$. Specifically we have a rational map $$\pi: Y \dashrightarrow \bP^{n}, \,\,\, [y_0,\dots ,y_n, y_{n+1}] \mapsto [y_0,\dots, y_n,0]$$ which has degree $2$ generically. Clearly $\sigma$ is the involution relative to $\pi$. To resolve the indeterminacy locus let us blow up $Y$ at $p$. Note that we view $X$ and $H$ as subvarieties of the linear subspace $(y_{n+1}=0) \cong \bP^{n}$. 

\begin{proposition} \label{Y projection}
Notations as above. After blowing up the point $p=[0,\dots,0,1]$ we obtain a morphism $\pi: \mathrm{Bl}_pY \rightarrow \bP^{n}$ which is also the blowup $\Bl_{X \cap H}Z \rightarrow Z$ of the double cover $Z$ of $\bP^{n}$ ramified over $X+H$ along the reduced inverse image of $X \cap H$.   
\begin{equation*}
\begin{tikzcd}
\Bl_{p}Y \arrow{rr}{\cong} \arrow[swap]{dr}{\pi}  \arrow[loop left]{l}{\sigma}
               &&\Bl_{X \cap H}Z \arrow{dl}{}\\
& \bP^{n} &
\end{tikzcd}
\end{equation*}
\end{proposition}
\begin{proof}
The equation of $Y$ is $f(y_0, \dots, y_n) + l(y_0, \dots, y_n)y_{n+1}^2 = 0$. After choosing the open affine $(y_{n+1} \neq 0)$ containing the point $p=[0,\dots,0,1]$ we assume that $y_{n+1}=1$. The blowup of the open affine at the point $(0,\dots,0) \in \bC^{n+1}$ is the morphism $\proj \bC[y_0,\dots,y_n][A_0,\dots,A_n]/(y_iA_j=y_jA_i) \rightarrow \spec \bC[y_0,\dots,y_n]$. The projection morphism $\pi$ is defined by $(y_0, \dots, y_n)[A_0,\dots,A_n] \mapsto [A_0,\dots,A_n]$. Let us describe the exceptional divisor. Choose an open subset $(A_0 \neq 0)$. Write $a_i = \frac{A_i}{A_0}$ for $1 \leq i \leq n$. Now pull back the defining equation of $Y$ using $y_i=y_0a_i$. We get $y_0(y_0^2f(1,a_1, \dots, a_n)+l(1,a_1, \dots, a_n))=0$. The blowup $\mathrm{Bl}_pY$ is given locally by $y_0^2f(1,a_1, \dots, a_n)+l(1,a_1, \dots, a_n)=0$. Over the point $p$ we have $y_0 = 0$ and hence $l(1,a_1, \dots, a_n) = 0$. One has similar descriptions when choosing different open subsets. It follows that globally the exceptional divisor is defined by $l(A_0,\dots,A_n) = 0$ which is mapped onto the hyperplane $H$ by $\pi$. Over the complement $\bP^{n} \backslash (X \cap H)$ of $X \cap H$ the morphism $\pi$ is a double cover branched along $(X+H) \backslash (X \cap H)$. The fibers of $\pi$ over the cubic surface $X \cap H$ are $\bP^1$'s (N.B. the involution $\sigma$ acts on these fibers). Because $\pi^{-1}(X \cap H)$ is a divisor in $\Bl_pY$, there exists a unique morphism $\Bl_pY \rightarrow \Bl_{(X \cap H)}\bP^{n}$ factoring $\pi$. It is not difficult to see that the morphism $\Bl_pY \rightarrow \Bl_{(X \cap H)}\bP^{n}$ is a double cover ramified over the strict transforms $X' + H'$ of $X + H$. Let $Z$ be the double cover of $\bP^{n}$ branched along $X+H$. By \cite[\S3]{CvS_defcover} $\Bl_pY$ is the blowup of $Z$ along the locus over the cubic surface $X \cap H$. \end{proof}

\begin{remark}
The cubic $n$-fold $Y$ and the double cover $Z$ are birational. Let us describe the birational maps. To get $Z$ one blows up $Y$ at the Eckardt point $p$ and then blow down the ruling of the associated cone. Conversely, by the proof of \propositionref{Y projection} the cubic $n$-fold $Y$ can be obtained by blowing down the strict transform of the inverse image of $H$ in $Z$ (to the point $p$) for $\Bl_{X\cap H}Z$. 
\end{remark}

\begin{remark} \label{weighted Z}
Let $Z$ be an $n$-dimensional quasi-smooth hypersurface of degree $6$ in $\bP(1,2,\dots,2,3)$. Let $z_0, z_1, \dots, z_n, z_{n+1}$ be the weighted homogeneous coordinates. After a change of coordinates one can assume that $Z$ has the equation $$z_{n+1}^2 + g(z_0^2, z_1, \dots, z_n) = 0$$ where $g$ is a cubic homogeneous polynomial. It is not difficult to see that $Z$ is a double cover of $\bP^{n}$ along a cubic $(n-1)$-fold $X$ union a hyperplane $H$. Moreover, both $X$ and $X \cap H$ are smooth. The cubic $n$-fold $Y$ constructed from the pair $(X,H)$ is obtained as the closure of $Z$ under the Veronese map of degree 3 from $\bP(1,2,\dots,2,3)$ to $\bP^{n+1}$.
\end{remark}

In what follows we shall focus on the case when $n=4$, that is, $X$ is a cubic threefold and $H$ is a hyperplane in $\bP^4$. The cubic fourfold $Y$ coming from $(X,H)$ contains an Eckardt point $p=[0,0,0,0,0,1]$ and admits an involution $\sigma$ fixing $p$ and a hyperplane (see (\ref{sigma})). It is also birational to the double cover $Z$ of $\bP^4$ branched along $X+H$ as described in \propositionref{Y projection}. 

To conclude this section, let us mention another geometric property of the smooth cubic fourfold $Y$. Still, $X$ and $H$ are considered as subvarieties of the hyperplane $(y_5=0) \cong \bP^4$. Choose a line $m$ on the cubic surface $X \cap H$ and let $P \subset Y$ be the plane generated by $m$ and the Eckardt point $p$. Choose a plane $P'$ which is complementary to $m$ in the hyperplane $(y_5=0) \cong \bP^4$ and project $Y$ from $P$ to $P'$. After blowing up $P$ inside $Y$ we obtain a quadratic bundle $\Bl_PY \rightarrow P'$ whose discriminant locus is a degree $6$ curve (cf. \cite{Voisin_cubic}). The sextic discriminant curve consists of two irreducible components: a quintic curve $C$ and a line $L$. Indeed, the quintic curve $C$ is the discriminant locus of the conic bundle obtained by projecting $X$ from the line $m$ to $P'$ in $(y_5=0) \cong \bP^4$ and the line $L$ is the intersection of $H$ and $P'$. The proof is left to the reader. 

\begin{remark}
More generally, there exists a finite correspondence between the moduli of pairs $(X,H)$ and $(C,L)$ (obtained by projecting $X$ from a line on $X \cap H$, see for example \cite[\S 4.2]{Beauville_determinantal}). As discussed in \cite{Laza_n16}, the moduli of $(C,L)$ is associated to the singularity $N_{16}$, and is birational to a locally symmetric variety of type IV. In conclusion, there is a close relation between the deformations of $N_{16}$ and $O_{16}$, and we expect the moduli of $(X,H)$ is also birational to a type IV locally symmetric domain.
\end{remark}

\section{A lattice theoretical characterization of cubic fourfolds with Eckardt points} \label{Y lattice}
We now study the smooth cubic fourfold $Y$ constructed in Section \ref{associated cubic4} from the perspective of Hodge theory. The cubic fourfold $Y$ is a \emph{special cubic fourfold} (meaning it contains an algebraic surface which is not homologous to a complete intersection, see \cite{Hassett_cubic}). In fact, $Y$ contains more algebraic cycles. There exists a positive definite lattice $M$ of rank $7$ and a primitive embedding $M \subset H^4(Y, \bZ) \cap H^{2,2}(Y) \subset H^4(Y, \bZ)$. Moreover, the sublattice $M$ is invariant with respect to the natural involution $\sigma$ (see (\ref{sigma})). We shall determine $M$ and its orthogonal complement $T=M_{H^4(Y,\bZ)}^{\perp}$ using lattice theory (in particular, we use the terminologies and notations of \cite{Nikulin_lattice}) and show that $Y$ can be characterized by the condition that it is ``$M$-polarized".

Let $Y$ be the smooth cubic fourfold (see (\ref{Y eqn})) coming from a pair $(X, H)$ where $X$ is a smooth cubic threefold and $H$ is a transverse hyperplane. Let us first describe how the involution $\sigma$ in (\ref{sigma}) decomposes the Hodge structure on the primitive cohomology $H^4_{\mathrm{prim}}(Y, \bQ)$.
\begin{lemma} \label{sigma decomposition}
The Hodge structure $H^4_{\mathrm{prim}}(Y, \bQ)$ splits as $$H^4_{\mathrm{prim}}(Y, \bQ) = H^4_{\mathrm{prim}}(Y, \bQ)_{+1} \oplus H^4_{\mathrm{prim}}(Y, \bQ)_{-1}$$ where $H^4_{\mathrm{prim}}(Y, \bQ)_{+1}$ (resp. $H^4_{\mathrm{prim}}(Y, \bQ)_{-1}$) is the eigenspace of the induced action $\sigma^*$ with eigenvalue $+1$ (resp. $-1$). The Hodge numbers of $H^4_{\mathrm{prim}}(Y, \bQ)_{+1}$ (resp $H^4_{\mathrm{prim}}(Y, \bQ)_{-1}$) are $[0,0,6,0,0]$ (resp. $[0,1,14,1,0]$). 
\end{lemma}
\begin{proof}
It suffices to check the claim on Hodge numbers for a particular cubic fourfold $Y$. We assume that $Y$ is given by $F(y_0, \dots, y_4) =y_0^3 + y_1^3 + y_2^3 + y_3^3 + y_4^3 + y_0y_5^2$. The Jacobian ideal is $J(F) = \langle 3y_0^2 + y_5^2, 3y_1^2, 3y_2^2, 3y_3^2, 3y_4^2, 2y_0y_5 \rangle$. Write $\Omega = \sum_i (-1)^i y_i dy_0 \wedge \dots \wedge \widehat{dy_i} \wedge \dots \wedge dy_5$. One has $\sigma^*\Omega = - \Omega$. A basis of $H^{3,1}(Y)$ (resp. $H^{1,3}(Y)$) is given by $\{\Res\frac{\Omega}{F^2}\}$ (resp. $\{\Res\frac{(y_1y_2y_3y_4y_5^2) \cdot \Omega}{F^4}\}$) which is clearly anti-invariant for the action of $\sigma^*$. Similarly, a basis of $H^{2,2}_{\mathrm{prim}}(Y)$ consists of classes of the form $\Res\frac{A \cdot \Omega}{F^3}$ with $A$ a cubic monomial in the Jacobian ring $R(F)$ (i.e. $A \in R(F)_3$). If $A$ contains an even power of $y_5$ (there are ${5 \choose 3} + {4 \choose 1} = 14$ such monomials) then $\Res\frac{A \cdot \Omega}{F^3}$ is anti-invariant for $\sigma^*$. Otherwise (there are ${4 \choose 2} = 6$ such $A$'s) $\Res\frac{A \cdot \Omega}{F^3}$ is $\sigma^*$-invariant.
\end{proof} 

The class of a hyperplane section of $Y$ is clearly fixed by the induced involution $\sigma^*$ (e.g. consider the hyperplane $(y_5=0)$). Let $M \subset H^4(Y, \bZ) \cap H^{2,2}(Y)$ be the $\sigma^*$-invariant sublattice of $H^4(Y, \bZ) \cap H^{2,2}(Y)$. By \lemmaref{sigma decomposition} and Hodge-Riemann relations, $M$ is a positive definite lattice of rank $7$. It is also easy to see that $M$ is saturated (i.e. $M = (M \otimes \bQ) \cap (H^4(Y, \bZ) \cap H^{2,2}(Y))$) and hence a primitive sublattice (i.e. $(H^4(Y, \bZ) \cap H^{2,2}(Y))/M$ is torsion free). 

We construct some algebraic surfaces in $Y$ whose cohomology classes belong to $M$. Let $p=[0,0,0,0,0,1]$ be the Eckardt point of $Y$. Recall that $Y$ contains a cone over the cubic surface $X \cap H$ with vertex $p$. It is classically known that the smooth cubic surface $X \cap H$ is isomorphic to $\bP^2$ with $6$ points blown up. Let $\bar{F}_0$ be the pull-back of a (general) line on $\bP^2$. Denote the exceptional curves by $\bar{F}_1, \dots, \bar{F}_6$. (The embedding of the blowup of $\bP^2$ at $6$ points into $\bP^3$ as a cubic surface is given by the linear system $|3\bar{F}_0-\bar{F}_1 - \dots - \bar{F}_6|$.) We shall use $F_0, F_1, \dots, F_6$ to denote the cones over the corresponding curves on $X \cap H$ with vertex $p$. In particular, $F_1, \dots, F_6$ are planes in $Y$. Also let the corresponding surface classes be $[F_0], [F_1], \dots, [F_6]$. 

\begin{lemma}
The surface classes $[F_0], [F_1], \dots, [F_6]$ are linearly independent in $H^4(Y, \bZ) \cap H^{2,2}(Y)$. Moreover, they are left invariant by $\sigma^*$.
\end{lemma}
\begin{proof}
The equations of the cones $F_0, F_1, \dots, F_6$ only involves $y_0, \dots, y_4$. Therefore, the involution $\sigma: [y_0, \dots, y_4, y_5] \mapsto [y_0, \dots, y_4, -y_5]$ fixes them. Now we use \propositionref{Y projection} to prove $[F_0], [F_1], \dots, [F_6]$ are linearly independent. Let $f: Z \rightarrow \bP^4$ be the double cover of $\bP^4$ branched along $X+H$. Set $\varphi: \Bl_{X \cap H}Z \rightarrow Z$ (resp. $\psi: \Bl_{X \cap H} \bP^4 \rightarrow \bP^4$) to be the blowup of $Z$ (resp. $\bP^4$) along the locus ramified over the cubic surface $X \cap H$ (resp. the surface $X \cap H$). Let $j: F \hookrightarrow \Bl_{X \cap H}\bP^4$ be the exceptional divisor of $\Bl_{X \cap H}\bP^4$. Note that $\psi^*(X+H) \sim X'+H'+2F$ where $X'$ (resp. $H'$) is the strict transform of $X$ (resp. $H$). By \cite[\S3]{CvS_defcover} we have the following commutative diagram 
\begin{equation} \label{blow up cubic surface}
\begin{tikzcd}
\Bl_{X \cap H}Z \arrow{rr}{\varphi} \arrow{dd}{g} \arrow{rrdd}{\phi}
               &&Z \arrow{dd}{f}\\
&& \\
\Bl_{X \cap H}\bP^4 \arrow{rr}{\psi} & &\bP^4
\end{tikzcd}
\end{equation}
where $g: \Bl_{X \cap H}Z \rightarrow \Bl_{X \cap H} \bP^4$ is the double cover of $\Bl_{X \cap H} \bP^4$ along $X'+H'$. By \cite[Thm. 7.31]{Voisin_hodge1} there exists an isomorphism of Hodge structures $H^4(\bP^4, \bZ) \oplus H^2(X \cap H, \bZ)(-1) \stackrel{\psi^*+j_*\circ\psi_{|_F}^*}{\longrightarrow} H^4(\Bl_{X \cap H}\bP^4, \bZ)$. The covering map $g$ induces an injective homomorphism $g^*: H^4(\Bl_{X \cap H}\bP^4, \bZ) \rightarrow H^4(\Bl_{X \cap H}Z, \bZ)$. By \propositionref{Y projection} $\Bl_{X \cap H}Z$ is also the blown-up $\Bl_pY$ of $Y$ at the Eckardt point $p$ (the exceptional divisor is $g^{-1}(H')$). Thus we have $H^4(\Bl_{X \cap H}Z, \bZ) \cong H^4(Y, \bZ) \oplus H^0(p, \bZ)(-2)$. By our construction, the classes represented by $\bar{F}_0, \bar{F}_1, \dots, \bar{F}_6$ are linearly independent in $H^2(X \cap H, \bZ)$. The strict transform of $F_i$ ($0 \leq i \leq 6$) in $\Bl_pY = \Bl_{X \cap H}Z$ is $g^{-1}(\psi^{-1}(\bar{F}_i))$. It follows that $[F_0], [F_1], \dots, [F_6]$ are linearly independent in $H^4(Y, \bZ)$.
\end{proof}

Let us consider the rank $7$ sublattice $N \subset M$ generated by $[F_0], [F_1], \dots, [F_6]$. Later we will show that $N$ is a saturated sublattice of $H^4(Y, \bZ) \cap H^{2,2}(Y)$ and hence $N=M$.
\begin{lemma} \label{N gram}
The Gram matrix of the lattice $N$ is given as follows.
\begin{enumerate}
\item $[F_i] \cdot [F_i] = 3$ for $1 \leq i \leq 6$.
\item $[F_0] \cdot [F_0]=7$.
\item $[F_0]  \cdot [F_i] = 3$ and $[F_i] \cdot [F_j] = 1$ for $1 \leq i \neq j \leq 6$.
\end{enumerate}
In particular, the lattice $N$ is positive definite and $\mathrm{discr}(N) = 64$. Furthermore, the square $h^2$ of the hyperplane class $h$ can be expressed as $3[F_0]-[F_1]- \dots -[F_6]$ in $H^4(Y, \bZ)$.
\end{lemma} 
\begin{proof}
The classes $[F_1], \dots, [F_6]$ are represented by planes in $Y$. The class $[F_0]$ corresponds to a union of planes $P_0 + P_0' + P_0''$ where $P_0$ share a line with $P_0'$ (resp. $P_0''$) and the intersection $P_0' \cap P_0''$ is a point (see for example \cite{Reid_surface}). If two planes contained in a cubic fourfold intersect along a line then their intersection number is $-1$ (cf. \cite{Voisin_cubic}). Also, the self intersection of a plane in a cubic fourfold equals $3$. The Gram matrix of $N$ can be computed easily using these observations.  

Alternatively, we compute the intersection numbers using \propositionref{Y projection}. Denote the blowup of the point $p \in Y$ by $q: \Bl_pY \rightarrow Y$. Let $i: E \hookrightarrow \Bl_pY$ be the exceptional divisor. By the blow-up formula (cf. \cite[Thm. 6.7, Cor. 6.7.1]{Fulton_intersection}) we have $q^*[F_j] = [F_j'] + i_*[V]$ ($1 \leq j \leq 6$) and $q^*[F_0] = [F_0'] + 3i_*[V]$ where $[F_j']$ and $[F_0']$ denote the corresponding proper transforms and $V$ is a $2$-dimensional linear space of $E \cong \bP^3$. Using \propositionref{Y projection} and Diagram (\ref{blow up cubic surface}) one computes that $[F_0'] \cdot [F_0'] = -2$, $[F_j'] \cdot [F_j'] = 2$ and $[F_0'] \cdot [F_j'] = 0$. Now we get $[F_j] \cdot [F_j] = q^*[F_j] \cdot q^*[F_j] = [F_j'] \cdot [F_j'] + 2[F_j'] \cdot i_*[V] + i_*[V] \cdot i_*[V] =2 + 2 \times 1 -1 = 3$ (see \cite{Fulton_intersection} Example 8.3.9). Similarly, one can verify that $[F_0]\cdot[F_0]=7$ and $[F_0] \cdot [F_j] = 3$. It is clear that $[F_j] \cdot [F_{j'}] = 1$ for $1 \leq j \neq j' \leq 6$.  

Take another general linear form $l'(y_0, \dots, y_4)$ and consider the surface $D$ in $Y$ cut out by $l=l'=0$. Clearly, the surface $D$ represents $h^2$ where $h$ is the hyperplane class of $Y$. Because $D$ is a cone (with vertex $p$) over the hyperplane section of the cubic surface $X \cap H$, we have $h^2 = [D] = 3[F_0]-[F_1]- \dots -[F_6]$. 
\end{proof}

\begin{remark}
A direct computation shows the following:
\begin{itemize}
\item $h^2 \cdot [F_0] = (3[F_0]-[F_1]- \dots -[F_6]) \cdot [F_0] = 3$;
\item $h^2 \cdot [F_i] =(3[F_0]-[F_1]- \dots -[F_6]) \cdot [F_i] = 1$ for $1 \leq i \leq 6$; 
\item $h^2 \cdot h^2 =(3[F_0]-[F_1]- \dots -[F_6]) \cdot (3[F_0]-[F_1]- \dots -[F_6]) = 3$. 
\end{itemize}
In particular, the cubic fourfold $Y$ is special (cf.\cite{Hassett_cubic}) with discriminant $8$ or $12$. 
\end{remark}

To show that $N$ is a saturated sublattice of $H^4(Y,\bZ) \cap H^{2,2}(Y)$ we need to compute its discriminant group. By \lemmaref{N gram} the Gram matrix $G_N$ of $N$ (with respect to $[F_0], [F_1], \dots, [F_6]$) is given by the following matrix.
$$
G_N=
\begin{pmatrix} 
7 & 3 & 3 & 3 & 3 & 3 & 3 \\
3 & 3 & 1 & 1 & 1 & 1 & 1 \\
3 & 1 & 3 & 1 & 1 & 1 & 1 \\
3 & 1 & 1 & 3 & 1 & 1 & 1 \\
3 & 1 & 1 & 1 & 3 & 1 & 1 \\
3 & 1 & 1 & 1 & 1 & 3 & 1 \\
3 & 1 & 1 & 1 & 1 & 1 & 3 \\
\end{pmatrix}
$$
The inverse Gram matrix $G_N^{-1}$ is given as follows which allows us to compute the dual lattice $N^* \subset N \otimes \bQ$ and the discriminant group $A_N = N^*/N$. Specifically, we consider the linear combinations of $[F_0], [F_1], \dots, [F_6]$ with coefficients given by the row vectors (or the column vectors) of $G_N^{-1}$. Denote them by $[F_0]^*, [F_1]^*, \dots, [F_6]^*$ respectively. Note that $[F_i] \cdot [F_j]^* = \delta_{ij}$. By abuse of notation, $[F_0]^*, [F_1]^*, \dots, [F_6]^*$ also denote the corresponding elements in $A_N = N^*/N$. It is straightforward to check that the discriminant group $A_N$ is isomorphic to $(\bZ/2\bZ)^6$ and $\{[F_1]^*, \dots, [F_6]^*\}$ is a basis. 
$$
G_N^{-1}=
\renewcommand\arraystretch{1.2}
\begin{pmatrix} 
4 & -\frac32 & -\frac32 & -\frac32 & -\frac32 & -\frac32 & -\frac32 \\
-\frac32 &1 & \frac12 & \frac12 & \frac12 & \frac12 & \frac12 \\
-\frac32 & \frac12 & 1 & \frac12 & \frac12 & \frac12 & \frac12 \\
-\frac32 & \frac12 & \frac12 & 1 & \frac12 & \frac12 & \frac12 \\
-\frac32 & \frac12 & \frac12 & \frac12 & 1 & \frac12 & \frac12 \\
-\frac32 & \frac12 & \frac12 & \frac12 & \frac12 & 1 & \frac12 \\
-\frac32 & \frac12 & \frac12 & \frac12 & \frac12 & \frac12 & 1 \\
\end{pmatrix}
$$

\begin{lemma} \label{disc N}
Notations as above. The discriminant group $A_N$ of $N$ is isomorphic to $(\bZ/2\bZ)^6$. The discriminant bilinear form $$b_N: A_N \times A_N \rightarrow \bQ/\bZ$$ is given by the following matrix (with respect to the basis $\{[F_1]^*, \dots, [F_6]^*\}$ of $A_N$).
$$
\renewcommand\arraystretch{1.2}
\begin{pmatrix} 
0 & \frac12 & \frac12 & \frac12 & \frac12 & \frac12 \\
 \frac12 & 0 & \frac12 & \frac12 & \frac12 & \frac12 \\
 \frac12 & \frac12 & 0 & \frac12 & \frac12 & \frac12 \\
 \frac12 & \frac12 & \frac12 & 0 & \frac12 & \frac12 \\
 \frac12 & \frac12 & \frac12 & \frac12 & 0 & \frac12 \\
 \frac12 & \frac12 & \frac12 & \frac12 & \frac12 & 0 \\
\end{pmatrix}
$$
\end{lemma}  

Now we show that $N \subset H^4(Y,\bZ) \cap H^{2,2}(Y)$ is saturated. As a result, we have $N= M$ and the natural embedding of $N$ into $H^4(Y, \bZ) \cap H^{2,2}(Y)$ is primitive.
\begin{proposition} \label{primitive embedding}
 The lattice $N$ generated by $[F_0], [F_1], \dots, [F_6]$ is saturated in $H^4(Y, \bZ) \cap H^{2,2}(Y)$.
\end{proposition}
\begin{proof}
Assume the natural embedding $N \subset H^4(Y, \bZ) \cap H^{2,2}(Y)$ is not saturated. Then it factors as $N \subsetneq \mathrm{Sat}(N) \subset H^4(Y, \bZ)$ where $\mathrm{Sat}(N)$ denotes the saturation of $N$ in $H^4(Y, \bZ)$. Thus, $\mathrm{Sat}(N)$ is a nontrivial overlattice of $N$ and corresponds to a nontrivial isotropic subgroup of $(A_N, b_N: A_N \times A_N \rightarrow \bQ/\bZ)$ (cf. \cite[Prop. 1.4.1]{Nikulin_lattice}). By \lemmaref{disc N} all the elements of $A_N$ are isotropic. These elements are linear combinations of $[F_1]^*, \dots, [F_6]^*$ with coefficients either $0$ or $1$. There are $6$ cases depending on the number of nonzero coefficients. For example let us suppose $[F_i]^*$ is contained in the isotropic subgroup. From the expression of $G_N^{-1}$ it is easy to see $[F_i]^* \equiv \frac12(h^2-[F_i]) \mod N$. It follows that the element $\frac12(h^2-[F_i])$ belongs to $H^4(Y, \bZ) \cap H^{2,2}(Y)$ (i.e. $h^2-[F_i]$ is divisible by $2$). The other cases are similar. In conclusion, $N \neq \mathrm{Sat}(N)$ if and only if some of the following elements are $2$-divisible in $H^4(Y, \bZ) \cap H^{2,2}(Y)$: 
\begin{enumerate}
\item $h^2 - [F_i]$;
\item $[F_i] - [F_j]$;
\item $[F_0]-[F_i]-[F_j]-[F_k]$;
\item $[F_i] +[F_j] -[F_k] -[F_l]$;
\item $[F_0] - [F_i]$;
\item $h^2-[F_0]$
\end{enumerate} 
for $1 \leq i,j,k,l \leq 6$. Let us do a case by case analysis.
\begin{enumerate}
\item Write $h^2 - [F_i] = 2x$ for some $x \in H^4(Y, \bZ)\cap H^{2,2}(Y)$. It is easy to see that $x\cdot h^2 = x^2=1$. Then $Y$ is a special cubic fourfold (cf. \cite[Def. 3.1.2]{Hassett_cubic}) labeled  by the rank $2$ lattice generated by $h^2$ and $x$. The corresponding discriminant $d =2$ (see op. cit. Definition 3.2.1). By op. cit. Section 4.4 this can not happen for the smooth cubic fourfold $Y$.
\item Write $[F_i] - [F_j] = 2x$ for some $x \in H^4(Y, \bZ)\cap H^{2,2}(Y)$. One computes that $x \cdot h^2 = 0$ (which implies that $x^2$ is even) and $x^2=1$. Clearly, this is a contradiction. 
\item This case can not happen. The argument is similar to that for Case (2).
\item Write $[F_i] +[F_j] -[F_k] -[F_l]=2x$ for some $x \in H^4(Y, \bZ)\cap H^{2,2}(Y)$. A direct computation shows that $h^2 \cdot x = 0$ and $x^2=2$. Then $Y$ is a special cubic fourfold with discriminant $6$. This is impossible because $Y$ is smooth (cf. \cite[\S 4.2]{Hassett_cubic}).
\item This case can not happen. The argument is similar to that for Case (1).
\item This case can not happen. The argument is similar to that for Case (2).
\end{enumerate} 
\end{proof}

We give a Hodge theoretic condition for smooth cubic fourfolds to admit Eckardt points.
\begin{proposition} \label{Hodge Eckardt}
Let $V \subset \bP^5$ be a smooth cubic fourfold. Suppose there exists a primitive embedding of $M \hookrightarrow H^4(V, \bZ) \cap H^{2,2}(V)$ sending $h^2 \in M$ to the square of the hyperplane class of $V$. Then $V$ contains an Eckardt point (and hence can be constructed from a cubic threefold and a hyperplane as in Section \ref{associated cubic4}). 
\end{proposition}
\begin{proof}
We first show that $V$ contains two planes which intersect at one point. Write the class of a hyperplane section of $V$ by $h_V$. Denote the class corresponding to $[F_i] \in M$ by $[F_i]_V$ ($0 \leq i \leq 6$). Consider (the saturation of) the rank two lattice generated by $h_V^2$ and $[F_i]_V$ ($1 \leq i \leq 6$). Clearly, the discriminant is $8$. By \cite[\S3]{Voisin_cubic} (see also \cite[\S4.1]{Hassett_cubic}) $V$ contains a plane $P_i$ whose class is $[F_i]_V$. Choose $i,j$ with $1 \leq i \neq j \leq 6$. Because $[F_i]_V \cdot [F_j]_V=1$, the corresponding planes $P_i$ and $P_j$ intersect at one point. (In a similar way one can show that $V$ contains at least $27$ planes.)

Choose two such planes $P$ and $P'$. Suppose $P \cap P' =\{p\}$. Consider the linear subspace $\Pi \cong \bP^4$ generated by $P$ and $P'$ in $\bP^5$. (In fact, the $27$ planes in $V$ are all contained in the hyperplane $\Pi$.) The cubic threefold $V \cap \Pi$ contains both $P$ and $P'$. Let $x_0, \dots, x_5$ be the homogeneous coordinates on $\bP^5$. Without loss of generality we assume that $P=(x_0=x_1=x_2=0)$ and $P'=(x_0=x_3=x_4=0)$. Note that $p= [0,0,0,0,0,1]$ and $\Pi = (x_0=0)$ (N.B. $x_1, \dots, x_5$ are considered coordinates on $\Pi$). Because $V \cap \Pi$ contains $P$, the equation of $V \cap \Pi \subset \Pi$ can be written as $x_3^2l_3+x_3q_3+x_4^2l_4+x_4q_4+x_5^2l_5+x_5q_5+c=0$ where $l_i$, $q_j$ and $c$ are homogeneous polynomials in $x_1$, $x_2$ of degree $1$, $2$ and $3$ respectively. But $P'$ is also contained in $V \cap \Pi$. It follows that $x_5^2l_5+x_5q_5+c$ is the zero polynomial. The equation of $V \cap \Pi \subset \Pi$ does not contain $x_5$, so $V \cap \Pi$ is a cone with vertex $p$ over a cubic surface (N.B. $V \cap \Pi$ can not be a cone with vertex a line because $V$ is smooth). Using \cite[Lem. 2.12]{CC_star} it is easy to show that $p$ is an Eckardt point of V. 
\end{proof}

Now let us focus on the lattice $M$ ($=N$) which is a positive definite lattice of rank $7$ generated by $[F_0], [F_1], \dots, [F_6]$. The natural embedding $M \subset H^4(Y, \bZ) \cap H^{2,2}(Y)$ is primitive. The Gram matrix $G_M (=G_N)$ and the discriminant group $(A_M, b_M) (=(A_N, b_N))$ have been computed in \lemmaref{N gram} and \lemmaref{disc N}. Note that $M$ is an odd lattice. However, the primitive part $M \cap H^4_{\mathrm{prim}}(Y, \bZ) = (h^2)_M^{\perp}$ is even and can be described explicitly.

\begin{proposition} \label{h2 perp}
The primitive part $M \cap H^4_{\mathrm{prim}}(Y, \bZ)$ of $M$ is an even lattice. More precisely, we have $M \cap H^4_{\mathrm{prim}}(Y, \bZ) \cong E_6(2)$.
\end{proposition}
\begin{proof} 
Because $(h^2)_{H^4(Y,\bZ)}^{\perp}$ is even (see for example \cite[Prop. 2.1.2]{Hassett_cubic}), the sublattice $(h^2)_M^{\perp}$ is even. We compute $M \cap H^4_{\mathrm{prim}}(Y, \bZ) = (h^2)_M^{\perp}$ explicitly. Let $x$ be an element of $M$ satisfying $x \cdot h^2 = 0$. Write $x = a_0[F_0] + a_1[F_1] + \dots + a_6[F_6]$ where $a_0, \dots, a_6$ are integers. Then $3a_0 + a_1+\dots+a_6 = 0$. It is not difficult to verify that $(h^2)_M^{\perp}$ is generated by $[F_i] - [F_j]$ and $[F_0]-[F_i]-[F_j]-[F_k]$ with $1 \leq i,j,k \leq 6$. Consider the $\bZ$-basis of $(h^2)_M^{\perp}$ consisting of $[F_2]-[F_1]$, $[F_3]-[F_2]$, $[F_4]-[F_3]$, $[F_5]-[F_4]$, $[F_6]-[F_5]$ and $-[F_0]+[F_1]+[F_2]+[F_3]$ (compare \cite[Thm. 8.2.12]{Dolgachev_cag}). The Gram matrix can be easily computed which coincides with the Gram matrix for $E_6(2)$. 
\end{proof}

\begin{remark}
\propositionref{h2 perp} shows that $M$ is closely related to the lattice $E_6$. This is not surprising because $M$ contains the classes of the planes generated by the Eckardt point $p$ and the $27$ lines on the cubic surface $X \cap H$. The configuration of these planes is determined by that of the $27$ lines. The relation between lines on a cubic surface and the $E_6$ lattice is classically known, see for example \cite[\S9.1]{Dolgachev_cag}.
\end{remark} 

Next we study the transcendental lattice $T:= M_{H^4(Y,\bZ)}^{\perp}$ which is the orthogonal complement of $M$ in $H^4(Y,\bZ)$. By \cite[Prop. 2.1.2]{Hassett_cubic} the middle integral cohomology lattice $H^4(Y,\bZ)$ is the odd unimodular lattice $I_{21,2} = (+1)^{\oplus 21} \oplus (-1)^{\oplus 2}$. We determine $T$ using analogues of the results in \cite[\S1]{Nikulin_lattice} for odd lattices. 
\begin{lemma} \label{T invariants}
The invariants of the transcendental lattice $T$ are computed as follows.
\begin{itemize}
\item $T$ is an even lattice of rank $16$. The signature is $(14,2)$.
\item $A_T \cong (\bZ/2\bZ)^6$. In particular, $T$ is $2$-elementary and $\mathrm{discr}(M) = 64$.
\item The discriminant quadratic form $q_T: A_T \rightarrow \bQ/2\bZ$ is isomorphic to the orthogonal direct sum $v \oplus v \oplus v$ where $v: (\bZ/2\bZ)^{\oplus 2} \rightarrow \bQ/2\bZ$ is the discriminant quadratic form of the $D_4$ lattice.
\end{itemize}
\end{lemma}
\begin{proof}
The lattice $T$ is contained in $(h^2)_{H^4(Y,\bZ)}^{\perp}$ and hence must be even. Clearly the signature is $(14,2)$. By \cite[Prop. 1.6.1]{Nikulin_lattice} we have $(A_T, b_T) \cong (A_M, -b_M)$ (as pointed out on Page 110 of op. cit. the results in Sections 1.4-1.6 hold for odd lattices after replacing discriminant quadratic forms by discriminant bilinear forms). By op. cit. Theorem 1.11.3 the signature $14-2 \pmod 8$ and $b_T$ determine $q_T$. Let $q$ be the finite quadratic form on $A_T \cong A_M$ given by the matrix in \lemmaref{disc N} (with respect to the basis $\{[F_1]^*, \dots, [F_6]^*\}$). It is straightforward to check that $b_T$ is the bilinear form of $q$ (cf. op. cit. Section 1.2). Also, $q$ equals $0$ for $28$ elements and equals $1$ for $36$ elements. Denote the quadratic form $u^{(2)}_+(2)$ (resp. $v^{(2)}_+(2)$) in  op. cit. Section 1.8 by $u$ (resp. $v$). Note that $v$ is isomorphic to the discriminant quadratic form of $D_4$. Since $q$ only takes values in integers, the only possibilities are $q \cong u \oplus u \oplus u \cong u \oplus v \oplus v$ or $q \cong v \oplus v \oplus v \cong u \oplus u \oplus v$ (cf. op. cit. Propositions 1.8.1 and 1.8.2). Because both $q$ and $v \oplus v \oplus v$ have Arf-invariant equal to $1$, we deduce that $q \cong v \oplus v \oplus v$. It follows that $q$ has signature $4 \pmod 8$ (see also op. cit. Theorem 1.3.3) and hence $q_T=q$.
\end{proof}

\begin{proposition} \label{adeT}
The lattice $T = M_{H^4(Y,\bZ)}^{\perp}$ is isomorphic to the orthogonal direct sum $U^{\oplus 2} \oplus D_4^{\oplus 3} $ 
and the natural map $O(T) \rightarrow O(q_T)$ is surjective. 
\end{proposition}
\begin{proof}
This follows from \lemmaref{T invariants}, \cite{Nikulin_lattice} Theorem 1.13.2 and Theorem 1.14.2 (see also Theorem 3.6.2).
\end{proof}

\begin{lemma} \label{qTe6}
The automorphism group $O(q_T)$ is isomorphic to the Weyl group $W(E_6)$ of the $E_6$ lattice.
\end{lemma}
\begin{proof}
By \cite{Bourbaki_lie4-6} Exercise VI.4.2 (see also \cite[\S 9.1.1]{Dolgachev_cag}) the action of $W(E_6)$ on $E_6/2E_6 \cong (\bZ/2\bZ)^6$ (equipped with the quadratic form $q(-) \equiv \frac12(-,-)_{E_6} \pmod{2\bZ}$) induces an isomorphism $W(E_6) \cong O(E_6/2E_6, q)$. The quadratic form $q$ has Arf-invariant $1$ and hence vanishes on $28$ vectors (cf. \cite[\S 9.1]{Dolgachev_cag}). As a result we get $q \cong q_T$ (see also the proof of \lemmaref{T invariants}). 
\end{proof}

Let us identify $O(b_M)$ and $O(b_T)$ using $(A_M, b_M) \cong (A_T, -b_T)$ and view $O(q_T) \subset O(b_T)$ as a subgroup of $O(b_M)$. The following lemma allows us to extend an automorphism of $T$ to an automorphism of $H^4(Y,\bZ)$ fixing the square of the hyperplane class $h^2$. 
\begin{lemma} \label{Me6}
The subgroup $O(q_T)  \subset O(b_T) \cong O(b_M)$ is the image of the stabilizer subgroup $O(M)_{h^2} \subset O(M)$ of $h^2$ under the natural map $O(M) \rightarrow O(b_M)$. Moreover, we have $O(M)_{h^2} \cong W(E_6)$. 
\end{lemma}
\begin{proof}
The idea is to consider the action of $W(E_6)$ on the $27$ lines of the cubic surface $X \cap H$ (and hence on the corresponding planes in $Y$ which generate the lattice $M$). Let us first set up some notations. The smooth cubic surface $X \cap H$ is isomorphic to $\bP^2$ with $6$ points blown up. Let $\bar{F}_0$ be the pull-back of a line on $\bP^2$. Denote the exceptional curves by $\bar{F}_1, \dots, \bar{F}_6$. (We shall use the same notations $\bar{F}_0, \bar{F}_1, \dots, \bar{F}_6$ to denote the corresponding curve classes.) The classes $\bar{F}_0, \bar{F}_1, \dots, \bar{F}_6$ form an orthonormal basis of $\Pic(X \cap H)$ which is isomorphic to $I_{1,6} = (+1) \oplus (-1)^{\oplus 6}$. By \cite[Thm. 8.2.12]{Dolgachev_cag} the vectors $\beta_1=-\bar{F}_0+\bar{F}_1+\bar{F}_2+\bar{F}_3$ and $\beta_j = \bar{F}_j - \bar{F}_{j-1}$ for $2 \leq j \leq 6$ form a canonical basis (op. cit. Definition 8.2.11) of the lattice $E_6(-1) = (3\bar{F}_0-\bar{F}_1-\dots-\bar{F}_6)_{\Pic(X \cap H)}^{\perp}$. The Weyl group $W(E_6)$ is generated by the reflections $r_{\beta_i}$ in the $(-2)$-vectors $\beta_i$. Let $p=[0,0,0,0,0,1]$ be the Eckardt point of $Y$. Set $[F_i]$ to be the surface class in $Y$ represented by the cone over $\bar{F}_i$ with vertex $p$ ($0 \leq i \leq 6$). Note that $\{[F_0], [F_1], \dots, [F_6]\}$ forms a basis of $M$. Every element $\bar{C}$ of $\Pic(X \cap H)$ is a linear combination of $\bar{F}_0, \bar{F}_1, \dots, \bar{F}_6$. We use $[C]$ to denote the linear combination of $[F_0], [F_1], \dots, [F_6]$ in $M$ with the same coefficients. In particular, the notation $[\beta_i]$ makes sense. 

We define a homomorphism $W(E_6) \rightarrow O(M)$ by $$r_{\beta_i} \mapsto s_{\beta_i}: [C] \mapsto [C]+ (\bar{C} \cdot \beta_i)[\beta_i].$$ (For example, because $\beta_1 \cdot \bar{F}_0 = -1$ on the cubic surface $X \cap H$ we have $s_{\beta_1}([F_0]) = [F_0] - [\beta_1] = 2[F_0]-[F_1]-[F_2]-[F_3]$.) With respect to the basis $\{[F_0], [F_1], \dots, [F_6]\}$ the transform $s_{\beta_1}$ is given by 
$$
\begin{pmatrix} 
2 & 1 & 1 & 1 & 0 & 0 & 0 \\
-1 & 0 & -1 & -1 & 0 & 0 & 0 \\
-1 & -1 & 0 & -1 & 0 & 0 & 0 \\
-1 & -1 & -1 & 0 & 0 & 0 & 0 \\
0 & 0 & 0 & 0 & 1 & 0 & 0 \\
0 & 0 & 0 & 0 & 0 & 1 & 0 \\
0 & 0 & 0 & 0 & 0 & 0 & 1 \\
\end{pmatrix}
$$
and $s_{\beta_j}$ is the transposition between $[F_{j-1}]$ and $[F_j]$ for $2 \leq j \leq 6$. It is straightforward to verify that $s_{\beta_i} \in O(M)$. It is also clear that the homomorphism $W(E_6) \rightarrow O(M)$ is injective. 

The reflections $s_{\beta_i}$ ($1 \leq i \leq 6$) fix $h^2 = 3[F_0]-[F_1]- \dots -[F_6]$. In other words, the image of $W(E_6)$ is contained in the stabilizer subgroup $O(M)_{h^2} \subset O(M)$ of $h^2$. We claim the image is $O(M)_{h^2}$. Let $g$ be an element of $O(M)_{h^2}$. Fix the basis $\{[F_0], [F_1], \dots, [F_6]\}$ for $M$ and the basis $\{\bar{F}_0, \bar{F}_1, \dots, \bar{F}_6\}$ for $\Pic(X \cap H)$. Let $\bar{g}: \Pic(X \cap H) \rightarrow \Pic(X \cap H)$ be the automorphism which has the same matrix as $g$ (with respect to the bases we choose). It suffices to show that $\bar{g}$ belongs to $W(E_6)$. The isometry $g$ fixes $h^2$. As a result, $\bar{g}$ fixes the anticanonical class $3\bar{F}_0-\bar{F}_1-\dots-\bar{F}_6$ of $X \cap H$. Recall that $\{\beta_1, \dots, \beta_6\}$ is a canonical basis in $E_6(-1)$. Because $(h^2)_{M}^{\perp} \cong E_6(2)$ (cf. \propositionref{h2 perp}) and $(3\bar{F}_0-\bar{F}_1-\dots-\bar{F}_6)_{\Pic(X \cap H)}^{\perp} \cong E_6(-1)$, $\{\bar{g}(\beta_1), \dots, \bar{g}(\beta_6)\}$ is also a canonical basis. By \cite{Dolgachev_cag} Theorem 8.2.12 and Corollary 8.2.15 there exists a unique $\bar{g}' \in W(E_6)$ so that $\bar{g}'$ fixes $3\bar{F}_0-\bar{F}_1-\dots-\bar{F}_6$ and $\bar{g}' = \bar{g}$ when restricting to the orthogonal complement $(3\bar{F}_0-\bar{F}_1-\dots-\bar{F}_6)_{\Pic(X \cap H)}^{\perp}$. It follows that $\bar{g} = \bar{g}'$ for a sublattice of finite index. Since $\Pic(X \cap H)$ is torsion free we get $\bar{g} = \bar{g}'$.   

The automorphisms $s_{\beta_i}$ induce $s_{\beta_i}^* \in O(b_M) \cong O(b_T)$. Specifically, $s_{\beta_1}^*$ corresponds to the matrix (with respect to the basis $\{[F_1]^*, \dots, [F_6]^*\}$) 
$$
\begin{pmatrix} 
1 & 0 & 0 & 0 & 0 & 0  \\
0 & 1 & 0 & 0 & 0 & 0  \\
0 & 0 & 1 & 0 & 0 & 0  \\
1 & 1 & 1 & 1 & 0 & 0  \\
1 & 1 & 1 & 0 & 1 & 0  \\ 
1 & 1 & 1 & 0 & 0 & 1 \\
\end{pmatrix}
$$
and $s_{\beta_j}^*$ ($2 \leq j \leq 6$) swaps $[F_{j-1}]^*$ with $[F_j]^*$. The quadratic form $q_T$ has been computed in the proof of \lemmaref{T invariants}. One easily verifies that $s_{\beta_i}^* \in O(q_T)$.  

Now it suffices to show that the homomorphism $W(E_6) \hookrightarrow O(M) \rightarrow O(q_T) \cong W(E_6)$ is an isomorphism. Note that $W(E_6)$ has a unique proper normal subgroup which has index $2$. It is then not difficult to prove that the kernel is trivial. Thus the homomorphism is an isomorphism. Alternatively, we claim that the composition $W(E_6) \hookrightarrow O(M) \rightarrow O(q_T) (\cong W(E_6))$ coincides with the isomorphism $W(E_6) \rightarrow O(E_6/2E_6, q(-) \equiv \frac12(-,-)_{E_6})$ (see \lemmaref{qTe6}). To prove this we use the following basis of $A_T$. 
$$
([F_1]^*, [F_2]^*, [F_3]^*, [F_4]^*, [F_5]^*, [F_6]^*)
\begin{pmatrix} 
0 & 1 & 0 & 0 & 0 & 0  \\
0 & -1 & 1 & 0 & 0 & 0  \\
0 & 0 & -1 & 1 & 0 & 0  \\
1 & 0 & 0 & -1 & -1 & 0  \\
-1 & 0 & 0 & 0 & 1 & 1  \\
1 & 0 & 0 & 0 & 0 & -1 \\
\end{pmatrix}
$$
The matrix of $s_{\beta_i}^*$ ($1 \leq i \leq 6$) under this basis coincides with the matrix corresponding to the action of $r_{\beta_i}$ on $E_6/2E_6$.  
\end{proof}

Another property of $T(\cong U^{\oplus 2}\oplus D_4^{\oplus 3})$ we want to mention is that it is the Milnor lattice of the singularity $\Sigma O_{16}$ (suspension of the $O_{16}$ singularity, which is the affine cone over a cubic surface; explicitly $\Sigma O_{16}$ is $V(f(x_1,\dots,x_4)+x_5^2)\subset (\bC^5,0)$, for a non-singular homogeneous cubic polynomial $f$). Using this fact and some results of Ebeling \cite{Ebeling_monodromy}, we get a good handle on the orthogonal group $O(T)$. To state the result, we need to recall the following standard notations:
\begin{notation} Denote by $O^+(T) \subset O(T)$ (resp. $\widetilde{O}(T) \subset O(T)$) the subgroup of isometries of $T$ with spinor norm $1$ (resp. the subgroup of isometries of $T$ that induce the identity on $A_T$). Also let $O^*(T) = O^+(T) \cap \widetilde{O}(T)$.
\end{notation} 
\begin{proposition} \label{Te6}
The lattice $T$ is isometric to the Milnor lattice associated to the singularity $\Sigma O_{16}$ and satisfies $O^*(T) =W(T)$ (the Weyl group, i.e. the group generated by reflections in $-2$ vectors in $T$). Geometrically, $O^*(T)$ is the local monodromy group of $\Sigma O_{16}$, $O^+(T)$ is the monodromy group for the pairs $(X, H)$ (consisting of a smooth cubic threefold $X$ and a hyperplane $H$ intersecting $X$ transversely), and $W(E_6) \cong O^+(T )/O^*(T)$ represents the monodromy at infinity (which acts by permuting the $27$ lines on the cubic surfaces $X \cap H$). 
\end{proposition}
\begin{proof}
The proof is similar to that of \cite[Prop. 4.22]{Laza_n16}. A singularity of type $\Sigma O_{16}$ is the germ given by $$(f(x_1, x_2,x_3,x_4)+x_5^2 = 0, 0) \subset (\bC^5, 0)$$ where $f$ is a homogeneous polynomial of degree $3$ (N.B. $V(f)\subset \bC^4$ defines the cone over a cubic surface, or equivalently the standard $O_{16}$ singularity). The singularity $\Sigma O_{16}$ is quasihomogeneous and has corank $4$ (see \cite[\S 11.1]{AGZV1}) and Milnor number $\mu=16$. Any singularity of class $\Sigma O_{16}$ has a $\mu$-constant deformation to $(x_1^3+x_2^3+x_3^3+x_4^3+x_5^2, 0)$. For this Fermat-like equation, 
the Milnor lattice can be computed using the theorems by Thom and Sebastiani and by Gabrielov (cf. \cite[\S 2.7]{AGZV2}). A routine computation allows us to identify the Milnor lattice for $\Sigma O_{16}$ with the lattice $T$ (N.B. it suffices to check that the invariants of the two lattices match). 

At this point, we can apply directly  \cite[Thm. 5.5]{Ebeling_monodromy} to the
singularity  $\Sigma O_{16}$ and the lattice $T$, and conclude that $O^*(T)$ is the local monodromy group of $\Sigma O_{16}$. Because the local monodromy group is generated by Picard-Lefschetz transformations in the vanishing cycles (see for example \cite[\S 2.3]{AGZV2}), we get $O^*(T) = W(T)$. 

Consider the natural exact sequence $$1 \rightarrow \widetilde{O}(T) \rightarrow O(T) \rightarrow O(q_T) \rightarrow 1.$$ Because $O(q_T) \cong W(E_6)$ (see \lemmaref{qTe6}) one gets $O(T )/\widetilde{O}(T) \cong W(E_6)$. Moreover, by \propositionref{adeT}, \lemmaref{Me6} and \cite[Cor. 1.5.2]{Nikulin_lattice} the exact sequence is right split and thus $O(T) = \widetilde{O}(T) \rtimes W(E_6)$. Similarly, we have $O^+(T )/O^*(T) \cong W(E_6)$ and $O^+(T) = O^*(T) \rtimes W(E_6)$. 

Let $\calU \subset \bP H^0(\bP^4, \calO(3)) \times \bP H^0(\bP^4, \calO(1))$ be the open subset parameterizing pairs $(X,H)$ with $X$ a smooth cubic threefold and $H$ a hyperplane intersecting $X$ transversely. Fix a base point which corresponds to a pair $(X,H)$. Let $Y$ be the associated cubic fourfold. There exists a monodromy action of $\pi_1(\calU)$ on $H^4(Y, \bZ)$ and on $T$. The global monodromy group $\Pi := \mathrm{Im}(\pi_1(\calU) \rightarrow \Aut(H^4(Y, \bZ)))$ is contained in $O^+(T)$ (cf. \cite{beauville_monodromy}). Now we show $O^+(T) \subset \Pi$. One has $O^*(T) \subset \Pi$ because $O^*(T)$ is generated by the reflections in vanishing cycles corresponding to the degenerations of $X$. Meanwhile, $\Pi$ acts as $W(E_6)$ on the set of lines on the cubic surface $X \cap H$ (and on the corresponding planes in $Y$). It follows that $\Pi = O^+(T)$.
\end{proof}

\begin{remark}
The arguments of Ebeling \cite{Ebeling_monodromy} are purely arithmetic, and thus it is possible to give a proof for Proposition \ref{Te6} without any reference to the $\Sigma O_{16}$ singularity. We have opted to include the arguments on the Milnor fiber of $\Sigma O_{16}$ as a way to give geometric meaning to the structure of the arithmetic group $O(T)$ (and the companion subgroups $O^*(T)$, etc.). 
\end{remark}

\section{A period map for cubic pairs $(X,H)$} \label{pairs torelli}
We have associated a smooth cubic fourfold $Y$ to a pair $(X,H)$ consisting of a smooth cubic threefold $X$ and a transverse hyperplane $H$. In this section, we shall define a period map for cubic pairs $(X,H)$ using the period map for the cubic fourfolds $Y$ and investigate the local and global Torelli problems. 

Let us first review Voisin's Torelli theorem \cite{Voisin_cubic} for smooth cubic fourfolds (see also \cite{Hassett_cubic}). Let $\calC_0$ be the moduli space for smooth cubic fourfolds (constructed using GIT). Denote by $\Lambda = I _{21,2} = (+1)^{\oplus 21} \oplus (-1)^{\oplus 2}$ the abstract lattice isomorphic to the integral middle cohomology of a cubic fourfold, by $h$ the class of a hyperplane section, and by $\Lambda_0 = (h^2)^{\perp}_{\Lambda} \cong A_2 \oplus E_8^{\oplus 2} \oplus U^{\oplus 2}$ the primitive cohomology. Write $$\calD :=\{\omega \in \bP(\Lambda_0 \otimes \bC) \,|\, \omega \cdot \omega = 0, \omega \cdot \bar{\omega} <0\}_0$$ (where the subscript $0$ indicates the choice of a connected component). Set $\Gamma :=\{\gamma \in O(\Lambda) \,|\, \gamma(h^2) = h^2\}$ and let $\Gamma^+ \subset \Gamma$ be the index $2$ subgroup stabilizing $\calD$ (i.e. $\Gamma^+$ consists of automorphisms of $\Lambda$ that preserve $h^2$ and the orientation of a negative definite $2$-plane in $\Lambda$).

\begin{theorem}[Torelli theorem for cubic fourfolds \cite{Voisin_cubic}] \label{Voisin torelli}
The period map for cubic fourfolds $\calC_0 \rightarrow \calD/\Gamma^+$ is an open immersion of analytic spaces. 
\end{theorem}

In our situation, we view $Y$ as ``lattice polarized" cubic fourfolds which are analogues of lattice polarized $K3$ surfaces (cf. \cite{Dolgachev_latticek3})
and have been used by Hassett \cite{Hassett_cubic}. Fix a sufficiently general pair $(X_b,H_b)$. Let $Y_b$ be the associated cubic fourfold which admits an involution $\sigma$ (see (\ref{sigma})). We have described the $\sigma^*$-invariant primitive sublattice $M \subset H^4(Y_b, \bZ)$ (which are generated by the classes $[F_0], [F_1], \dots, [F_6]$) in Section \ref{Y lattice}. Let us fix a primitive embedding $M \subset H^4(Y_b, \bZ) \cong \Lambda$. Write the image of $[F_i]$ in $\Lambda$ by $f_i$. In what follows, we identify $M$ with its image in $\Lambda$. In other words, $M \subset \Lambda$ will be considered as an abstract lattice spanned by $f_0, \dots, f_6$ (together with a primitive embedding into $\Lambda$). The intersection form of $M$ is given by the Gram matrix in \lemmaref{N gram}. In particular, $M$ contains $h^2$. Again $T = M_{\Lambda}^{\perp}$. Let $O^+(T) \subset O(T)$ (resp. $\widetilde{O}(T) \subset O(T)$) be the subgroup of isometries of $T$ preserving the orientation of a negative definite $2$-plane in $T$ (resp. the subgroup of isometries of $T$ that induce the identity on $A_T$). Set $O^*(T) = O^+(T) \cap \widetilde{O}(T)$. 

\begin{remark}
The embedding $M \subset \Lambda$ depends on the choice of a marking $H^4(Y_b, \bZ)\! \stackrel{\cong}{\rightarrow} \Lambda$. It also depends on the choice of an orthonormal basis $\{ \bar{F}_0, \bar{F}_1, \dots, \bar{F}_6\}$ of $\Pic(X_b \cap H_b) \cong I_{1,6}$ (i.e. the choices of the classes $[F_0], [F_1], \dots, [F_6]$). By \propositionref{adeT}, \lemmaref{Me6} and \cite[Cor. 1.5.2] {Nikulin_lattice} isometries of $M$ fixing $h^2$ extend to automorphisms of $\Lambda$. As a result, our choice of the embedding $M \subset \Lambda$ is unique up to the action of $O(\Lambda)$. See also \cite[Thm. 3.7]{Laza_maxirrat}.
\end{remark}

\begin{notation}
Let us introduce the following notations. 
\begin{itemize}
\item $\calM_0$: the moduli space of pairs $(X,H)$ consisting of a smooth cubic threefold $X$ and a transverse hyperplane $H$ (constructed as a Zariski open subset of the GIT quotient $\bP H^0(\bP^4, \calO_{\bP^4}(3)) \times \bP H^0(\bP^4, \calO_{\bP^4}(1)) \git \SL(5, \bC)$). Because such pairs $(X,H)$ are GIT stable, $\calM_0$ is a geometric quotient.
\item $\calD_M:= \{\omega \in \bP(\Lambda_0 \otimes \bC) \,|\, \omega \cdot \omega = 0, \omega \cdot \bar{\omega} <0, \omega \cdot M =0\}_0$. It is convenient to identify $\calD_M$ to the domain $\{\omega \in \bP(T \otimes \bC) \,|\, \omega \cdot \omega = 0, \omega \cdot \bar{\omega} <0\}_0$. Note that $O^+(T)$ acts on $\calD_M$ naturally. 
\item $\Gamma^+_M:=\{\gamma \in \Gamma^+ \,|\, \gamma(h^2)=h^2, \gamma(M) \subset M\}$. By \propositionref{adeT}, \lemmaref{Me6} and \cite[Cor.1.5.2]{Nikulin_lattice} $\Gamma^+_M$ can be identified with $O^+(T)$. 
\end{itemize}
\end{notation}


We define a period map for cubic pairs $(X,H) \in \calM_0$. Let $Y$ be the associated smooth cubic fourfold. After choosing an orthonormal basis of $\Pic(X \cap H) \cong I_{1,6}$ we get a primitive embedding $j: M \hookrightarrow H^4(Y,\bZ) \cap H^{2,2}(Y) \subset H^4(Y, \bZ)$ (N.B. $j$ sends $h^2$ to the square of the hyperplane class). To associate a period point on $\calD_M$ to $Y$, one needs to choose a marking $\phi: \Lambda \stackrel{\cong}{\rightarrow} H^4(Y, \bZ)$ satisfying $\phi|_M =j$. Now we consider $\phi_{\bC}^{-1}(H^{3,1}(Y))$ (where $\phi_\bC := \phi \otimes \bC$) which belongs to the subdomain $\calD_M$. One also has to take the action of $\Gamma_M^+$ (which can be identified with $O^+(T)$) on the markings $\phi$ into consideration. Specifically, suppose we choose a different orthonormal basis of $\Pic(X \cap H)$ which gives a different $j'$. Then $j$ and $j'$ are related by an element $g_M$ of the stabilizer group $O(M)_{h^2}$. By \propositionref{adeT}, \lemmaref{Me6} and \cite[Cor.1.5.2]{Nikulin_lattice} $g_M$ can be extended to an automorphism $g \in O(\Lambda)$ such that $g|_T \in O^+(T)$. The marking $\phi \circ g$ satisfies $(\phi \circ g)|_M = j'$. Moreover, it is not difficult to verify that the $O^+(T)$-orbit of the period point does not depend on the projective equivalence class of $(X,H)$. To sum up, we get the following period map (the choice of the monodromy group is also explained in \propositionref{Te6}) $$\calP_0: \calM_0 \rightarrow \calD_M/O^+(T).$$ 

Now let us prove the local Torelli and the global Torelli theorems for $\calP_0$.
\begin{proposition} \label{dP0 injective}
The period map $\calP_0: \calM_0 \rightarrow \calD_M/O^+(T)$ is a local isomorphism at every point of $\calM_0$.
\end{proposition}
\begin{proof}
It is enough to show the differential $d\calP_0$ is injective (N.B. $\dim \calM_0 = \dim \calD_M/O^+(T) = 14$). The infinitesimal deformations of hypersurfaces in projective spaces have been well studied (see for example \cite[\S5.4]{CM-SP}). Let $Y$ be the smooth cubic fourfold associated to a smooth cubic threefold $X$ and a transverse hyperplane $H$. Write the equation of $Y$ as $F=f(y_0,\dots,y_4)+l(y_0,\dots,y_4)y_5^2$. In particular, $X$ and $H$ are subvarieties of $(y_5=0) \cong \bP^4$ and are cut out in $\bP^4$ by $f=0$ and $l=0$ respectively. 
Let $X'$ (resp. $H'$) be a deformation of the hypersurface $X$ (resp. $H$) in $\bP^4$. Assume the equations for $X'$ and $H'$ are $f+\epsilon f'=0$ and $l+\epsilon l'=0$ respectively. We view $f'$ and $l'$ as elements of $\bC[y_0, \dots, y_4]_3/(\bC \cdot f)$ and $\bC[y_0, \dots, y_4]_1/(\bC \cdot l)$ (where $\bC[y_0, \dots, y_4]_d$ denotes the vector space of degree $d$ homogeneous polynomials) respectively. Notice that $f+\epsilon f' + (l+\epsilon l')y_5^2 = (f+ly_5^2)+\epsilon(f'+l'y_5^2)$. Suppose $d\calP_0(X',H')=\overrightarrow{0}$. From the local Torelli for cubic fourfolds (see for example \cite[\S0]{Voisin_cubic}) we deduce that $f'+l'y_5^2$ is contained in the Jacobian ideal $J(F)$ of $F$. A direct computation shows that $$f'=\frac{\partial f}{\partial y_0}l_0 + \dots +\frac{\partial f}{\partial y_4}l_4 \,\,\,\text{and}\,\,\, l' = \frac{\partial l}{\partial y_0}l_0 + \dots +\frac{\partial l}{\partial y_4}l_4 +2c \cdot l$$ where $l_k$ ($0 \leq k \leq 4$) are linear polynomials in $y_0, \dots, y_4$ and $c$ is a constant number. Thus $f' \in J(f)/(\bC \cdot f)$ and $l' \in J(l)/(\bC \cdot l)$. Moreover, the pencil spanned by $(X', H')$ and $(X,H)$ is tangent to the $\mathrm{GL}(5, \bC)$-orbit through $(X,H)$ (cf. \cite[Ex. 5.4.1]{CM-SP}). The injectivity of $d\calP_0$ then follows (see also the proof of \cite[Lem. 1.8]{ACT_cubic3}). 
\end{proof}

\begin{theorem} \label{P0 injective}
The period map $\calP_0: \calM_0 \rightarrow \calD_M/O^+(T)$ is an isomorphism onto its image.
\end{theorem}
\begin{proof}
Suppose the pairs $(X,H)$ and $(X', H')$ (denote the associated smooth cubic fourfolds by $Y$ and $Y'$ respectively) have the same image in $\calD_M/O^+(T)$. That is to say, there exist markings $\phi: \Lambda \rightarrow H^4(Y, \bZ)$ and $\phi': \Lambda \rightarrow H^4(Y', \bZ)$ and an isometry $g_T \in O^+(T)$ such that $g_T(\phi_\bC^{-1}(H^{3,1}(Y))) = \phi'^{-1}_\bC(H^{3,1}(Y'))$ (N.B. $\phi_\bC^{-1}(H^{3,1}(Y))$ and $\phi'^{-1}_\bC(H^{3,1}(Y'))$ both belong to $T$). \propositionref{adeT} and \lemmaref{Me6} (see also \cite[Cor.1.5.2]{Nikulin_lattice}) allow one to extend $g_T$ to an automorphism $g \in O(\Lambda)$. In particular, $g$ fixes $h^2$. Now we apply the global Torelli theorem for cubic fourfolds (see \theoremref{Voisin torelli}). It follows that the associated cubic fourfolds $Y$ and $Y' $ (which admit Eckardt points) are projectively equivalent. In other words, there exists a projective transform $\beta$ sending $Y$ to $Y'$. Generically, $Y$ (resp. $Y'$) admits one Eckardt point $p$ (resp. $p'$) (cf. \cite[Thm. 2.10, Thm. 4.2 and \S6]{CC_star}). (Using \lemmaref{Eckardt eqn} it is not difficult to see that the locus of smooth cubic fourfolds containing Eckardt points is irreducible, so we can speak of a generic such fourfold.) Consider the extended involution $\sigma$ (resp. $\sigma'$) (see (\ref{sigma})) on the blowup $\Bl_pY$ (resp. $\Bl_{p'}Y'$). Since $\beta$ carries the fixed locus of $\sigma$ (which is $X+H$ by \propositionref{Y projection}) to the fixed locus of $\sigma'$ (which is $X'+H'$), $(X,H)$ and $(X', H')$ are projectively equivalent. This proves that $\calP_0$ is generically injective. From \propositionref{dP0 injective} we know $\calP_0$ is a local isomorphism. As a result, $\calP_0$ is an isomorphism onto its image.
\end{proof}

\begin{remark}
By \cite[Prop. 2.2.3]{Hassett_cubic} the period map for cubic fourfolds $\calC_0 \rightarrow \calD/\Gamma^+$ is an algebraic map. A similar argument shows that $\calP_0: \calM_0 \rightarrow \calD_M/O^+(T)$ is also algebraically defined.
\end{remark}

\begin{remark}
Let $Z$ be a quasi-smooth hypersurface of degree $6$ in $\bP(1,2,2,2,2,3)$. Note that $Z$ is a double cover of $\bP^4$ branched over a smooth cubic threefold $X$ and a hyperplane $H$ intersecting $X$ transversely (cf. Remark \ref{weighted Z}). The isomorphism class of $Z$ is determined by the isomorphism class of the pair $(X,H)$ consisting of the branching data $X$ and $H$. Let $Y$ be the smooth cubic fourfold associated to $(X,H)$. The morphism $\Bl_pY \cong \Bl_{X \cap H} Z \rightarrow Z$ (see \propositionref{Y projection}) allows one to identify the Hodge structures on $H_{\mathrm{prim}}^4(Z, \bZ)$ and $T=M_{H^4(Y,\bZ)}^{\perp}$. A global Torelli theorem for weighted degree $6$ hypersurfaces in $\bP(1,2,2,2,2,3)$ can be derived from \theoremref{P0 injective} (compare \cite{Saito_wttorelli} and \cite{DT_wttorelli}). 
\end{remark}

\section{Special Heegner divisors in the period domain} \label{Heegner}
We introduce two Heegner divisors in the locally symmetric domain $\calD_M/O^+(T)$, namely, the nodal Heegner divisor $H_n$ and the tangential Heegner divisor $H_t$. We also establish a Borcherds' relation between the Hodge line bundle on $\calD_M/O^+(T)$ and these Heegner divisors.  

The locally symmetric domain $\calD_M/O^+(T)$ is associated to the transcendental lattice $T \cong U^{\oplus 2} \oplus D_4^{\oplus 3}$. The lattice theoretical invariants (e.g. the discriminant group $A_T$ and the discriminant quadratic form $q_T$) for $T$ have been computed in \lemmaref{T invariants}, \propositionref{adeT} and \lemmaref{qTe6}. We also need the following notations.

\begin{notation}
Let $v$ be an element of $T$. 
\begin{itemize}
\item $\mathrm{div}(v)$: the positive generator of the ideal $v \cdot T \subset \bZ$ (called the \emph{divisibility} of $v$). 
\item $\hat{v} :=v/\mathrm{div}(v)$ (viewed as an element in the discriminant group $A_T$).
\item $r_v: T \otimes \bQ \rightarrow T \otimes \bQ$: the reflection in the hyperplane $v^{\perp}$ (provided that $v$ is non isotropic) defined by $r_v(x) = x - 2\frac{x \cdot v}{v^2}v$. 
\end{itemize}
\end{notation}

For later use, let us choose a set of generators for $T \cong U^{\oplus 2} \oplus D_4^{\oplus 3}$. Denote a standard basis for the first copy of $U$ (resp. the second copy of $U$) by $e_1$ and $f_1$ (resp. $e_2$ and $f_2$). In particular, $e_i^2 =f_i^2=0$ and $e_i \cdot f_i = 1$ ($i=1,2$). The lattice $T$ contains $3$ copies of $D_4$. For the first copy of $D_4$, we let $\alpha_1,\dots, \alpha_4$ be the simple roots as in \cite[\S VI.4.8]{Bourbaki_lie4-6}. Note that $\alpha_1^2=\dots=\alpha_4^2=2$, $\alpha_1 \cdot \alpha_2 = \alpha_2 \cdot \alpha_3 = \alpha_2 \cdot \alpha_4 = -1$, and all the other intersection numbers are zero. The discriminant group $A_{D_4}$ is isomorphic to $(\bZ/2\bZ)^2$ which is  generated by $\frac12(\alpha_1+\alpha_3)$ and $\frac12(\alpha_1+\alpha_4)$ (N.B. $\frac12(\alpha_3+\alpha_4) = \frac12(\alpha_1+\alpha_3)+\frac12(\alpha_1+\alpha_4)$ in $A_{D_4}$). Also, $q_{D_4}(\frac12(\alpha_1+\alpha_3)) = q_{D_4}(\frac12(\alpha_1+\alpha_4)) = q_{D_4}(\frac12(\alpha_3+\alpha_4)) \equiv 1 \pmod{2\bZ}$. Similarly, we choose a basis $\beta_1, \dots, \beta_4$ (resp. $\gamma_1, \dots, \gamma_4$) for the second copy (resp. the third copy) of $D_4$. 

Before defining the Heegner divisors, we note the following lemmas.
\begin{lemma} \label{typeab}
Let $v \in T$ be a primitive vector. Suppose that the reflection $r_v$ preserves $T$ (i.e. $r_v(T) \subset T$) and $\calD_M \cap v^{\perp} \neq \emptyset$. Then there are two possibilities. 
\begin{enumerate}
\item $v^2=2$ and $\mathrm{div}(v)=1$. In this case, $\hat{v} =0$ in $A_T$.
\item $v^2=4$ and $\mathrm{div}(v)=2$. In this case, $q_T(\hat{v}) \equiv 1 \pmod{2\bZ}$.
\end{enumerate}
\end{lemma}
\begin{proof}
The assumption $r_v(T) \subset T$ is equivalent to $v^2|2\mathrm{div}(v)$. The intersection $\calD_M \cap v^{\perp}$ is empty if $v^2 <0$ (Hodge-Riemann bilinear relations). By the proof of \lemmaref{T invariants} the discriminant quadratic form $q_T$ only takes values in integers. As a result, we get $\mathrm{div}(v)^2 | v^2$ (consider $q_T(\hat{v})$). The lemma then easily follows.  
\end{proof}

\begin{lemma} \label{typea}
Let $v$ be a primitive vector of $T$ with $v^2=2$ and $\mathrm{div}(v)=1$. Then $v$ is unique up to the action of $O^*(T)$ (and hence also unique up to the action of $O^+(T)$). 
\end{lemma}
\begin{proof}
We use Eichler's criterion (cf. \cite[Prop. 3.3]{GHS_modularvar} or \cite[Prop. 1.3.2]{LOG_quartick3}). More precisely, let $u$ and $w$ be nonzero primitive vectors of an even lattice $L$ which contains $U \oplus U$. There exists an element $g \in O^*(L)$ such that $g(u)=w$ if and only if $u^2=w^2$ and $\hat{u} = \hat{w}$. The lemma is a direct application of Eichler's criterion (notice that $\hat{v} = 0$ in $A_T$).
\end{proof}

\begin{lemma} \label{typeb}
Consider primitive vectors $v \in T$ with $v^2=4$ and $\mathrm{div}(v)=2$. There are $36$ equivalence classes of $v$ for the action of $O^*(T)$. Moreover, these vectors form a single $O^+(T)$-orbit.   
\end{lemma}
\begin{proof}
We apply Eichler's criterion to prove the first assertion. The discriminant quadratic form $q_T: A_T \rightarrow \bQ/2\bZ$ vanishes on $28$ elements and equals $1 \pmod{2\bZ}$ for the rest $36$ elements (see the proof of \lemmaref{T invariants}). Notice that $q_T(\hat{v}) \equiv 1 \pmod{2\bZ}$ if $v^2=4$ and $\mathrm{div}(v)=2$. By Eichler's criterion, there are at most $36$ $O^*(T)$-equivalence classes for $v$. Now we construct a representative for every class. Let $\{e_1,f_1,e_2,f_2, \alpha_1,\dots, \alpha_4,\beta_1,\dots,\beta_4,\gamma_1,\dots, \gamma_4\}$ be a basis for $T \cong U^{\oplus 2} \oplus D_4^{\oplus 3}$ (see the beginning of this section). Let $i, j \in \{1,3,4\}$ be two distinct integers. Similarly for $k \neq l$ and $s \neq t$. Consider the $9$ primitive vectors of the form $\alpha_i + \alpha_j$, $\beta_k+\beta_l$ and $\gamma_s + \gamma_t$ respectively. Also consider the following $27$ primitive vectors: $2e_1-2f_1+\alpha_i + \alpha_j+\beta_k+\beta_l+\gamma_s + \gamma_t$. All these $36$ vectors $v$ satisfy $v^2=4$ and $\mathrm{div}(v)=2$. It is also easy to show that the corresponding $\hat{v} \in A_T$ are all different. In fact, these $\hat{v}$ are exactly the $36$ elements for which $q_T$ does not vanish.   

Let us prove the second assertion. Let $v_1$ and $v_2$ be primitive vectors of $T$ with $v_1^2=v_2^2=4$ and $\mathrm{div}(v_1)=\mathrm{div}(v_2)=2$. If $\hat{v}_1 = \hat{v}_2$ then $v_1$ and $v_2$ are $O^*(T)$-equivalent by Eichler's criterion. Suppose that $\hat{v}_1 \neq \hat{v}_2$ in $A_T$. By \lemmaref{qTe6} and its proof, we have $(E_6/2E_6, q)  \cong (A_T, q_T)$ (recall that $q(-) \equiv \frac12(-,-)_{E_6} \pmod{2\bZ}$) which induces $W(E_6) \cong O(q_T)$. Note that the $36$ elements of $E_6/2E_6$ for which $q$ does not equal to $0$ corresponds to the $36$ pairs of opposite roots of $E_6$. Because $W(E_6)$ acts transitively on the set of roots (see for example \cite[\S8.2]{Dolgachev_cag}), these $36$ elements form a single orbit. It follows that  there exists $\bar{g} \in O(q_T)$ such that $\bar{g}(\hat{v}_1) = \hat{v}_2$. Using the natural short exact sequence $1 \rightarrow O^*(T) \rightarrow O^+(T) \rightarrow O(q_T) \rightarrow 1$ we lift $\bar{g}$ to $g \in O^+(T)$. By Eichler's criterion, $g(v_1)$ and $v_2$
 are in the same orbit of $O^*(T)$. In other words, there is $f \in O^*(T)$ such that $f(g(v_1)) = v_2$.
\end{proof}

We introduce some Heegner divisors (cf. \cite[\S3]{Hassett_cubic} and \cite[\S1.3]{LOG_quartick3}) for $\calD_M/O^*(T)$ and $\calD_M/O^+(T)$. Let $\Pi \subset O^+(T)$ be a finite-index subgroup (in our situation, $\Pi$ is either $O^*(T)$ or $O^+(T)$). Let $\pi: \calD_M \rightarrow \calD_M/\Pi$ be the quotient map. For a nonzero $v \in T$, we write $$\calH_{v}(\Pi):= \bigcup_{g \in \Pi} g(v)^{\perp} \cap \calD_M, \,\,\,\,\,\, H_v(\Pi) :=\pi(\calH_v(\Pi)).$$ If $v^2>0$ then $\calH_v(\Pi)$ is not empty and gives a hyperplane arrangement. Also, $H_v(\Pi)$ is a prime divisor (which is also $\bQ$-Cartier as discussed in \cite[\S1.3.1]{LOG_quartick3}) in the locally symmetric variety $\calD_M/\Pi$. We call it the \emph{Heegner divisor} associated to $v$. Both $\calH_{v}(\Pi)$ and $H_{v}(\Pi)$ depend only on the $\Pi$-class of $v$. Following \cite{LOG_quartick3} we say $H_{v}(\Pi)$ is \emph{reflective} if the reflection $r_v$ belongs to $\Pi$.  

Consider primitive vectors $v \in T$ of either Type (1) (i.e. $v^2=2$ and $\mathrm{div}(v)=1$) or Type (2) (i.e. $v^2=4$ and $\mathrm{div}(v)=2$) as in \lemmaref{typeab}. By \lemmaref{typea} there is a single $O^*(T)$-orbit for Type (1) vectors. We denote the corresponding Heegner divisor $H_v(O^*(T))$ in $\calD_M/O^*(T)$ by $H_0$. Note that $H_0$ is reflective. There are $36$ $O^*(T)$-equivalence classes of Type (2) vectors (see \lemmaref{typeb}). Let us denote the corresponding Heegner divisors by $H_1, \dots, H_{36}$.

\begin{definition} \label{HnHt}
\leavevmode
\begin{enumerate}
\item Let $v$ be a vector of Type (1) as in \lemmaref{typeab}. The \emph{nodal Heegner divisor} $H_n$ in $\calD_M/O^+(T)$ is $H_v(O^+(T))$. 
\item Let $v$ be a vector of Type (2) as in \lemmaref{typeab}. The \emph{tangential Heegner divisor} $H_t$ in $\calD_M/O^+(T)$ is $H_v(O^+(T))$.
\end{enumerate}
(Thanks to \lemmaref{typea} and \lemmaref{typeb} the definitions make sense. Moreover, both $H_n$ and $H_t$ are reflective Heegner divisors.)
\end{definition}

\begin{remark}
Later we shall see that a generic point of $H_n$ (resp. $H_t$) corresponds to the a cubic pair $(X,H)$ with $X$ nodal (resp. $H$ simply tangent to $X$).
\end{remark}

\begin{remark}
The work of Borcherds, Bruinier (see \cite{Bruinier_picrk} and references therein) and the refinement given in \cite{BLMM_nl} allows us to compute the rank of the Picard group of $\calD_M/O^*(T)$ (more generally, the Picard rank of certain modular varieties of type IV). More specifically, let $S_{k, T}$ denote the space of (vector-valued) cusp forms of weight $k$ with values in $T$. Borcherds has defined a homomorphism $S_{8,T} \rightarrow \Pic(\calD_M/O^*(T)) \otimes {\bC}/\langle \lambda(O^*(T)) \rangle$ (where $\lambda(O^*(T))$ denotes the Hodge bundle, see the discussions below). Because $T$ contains two copies of $U$, the homomorphism is injective (cf. \cite{Bruinier_picrk}). By \cite{BLMM_nl} this is in fact an isomorphism. A formula for computing the dimension of $S_{8, T}$ is given by Bruinier \cite{Bruinier_picrk} (see also \cite[\S3.2]{LOG_quartick3}). Using \lemmaref{T invariants} we get $$\dim(S_{8,T}) = 64 + \frac{64 \cdot 8}{12} - 18-\frac{65}{3}-18-28 =21$$ and hence the Picard rank of $\calD_M/O^*(T)$ is $22$.
\end{remark}

What will be important for us is a relation between the Hodge bundle on $\calD_M/O^+(T)$ (resp. the Hodge bundle on $\calD_M/O^*(T)$) and the Heegner divisors $H_n$ and $H_t$ (resp. $H_0, H_1, \dots, H_{36}$) in $\Pic(\calD_M/O^+(T))_{\bQ}$ (resp. in $\Pic(\calD_M/O^*(T))_{\bQ}$). We call this type of relation a \emph{Borcherds' relation}. Let $\Pi$ be either $O^*(T)$ or $O^+(T)$. Recall that the \emph{Hodge bundle} $\lambda(\Pi)$ is the fractional line bundle on $\calD_M/\Pi$ defined as the quotient of $\lambda = \calO_{\calD_M}(-1)$ (the restriction of the tautological line bundle on $\bP(T \otimes \bC)$ to $\calD_M$) by $\Pi$. By abuse of notation, we also denote the divisor class of the Hodge bundle by $\lambda(\Pi)$ (which is $\bQ$-Cartier). The Baily-Borel compactification $(\calD_M/\Pi)^* = \Proj \bigoplus_{m \geq 0}H^0(\calD_M, \lambda^{\otimes m})^{\Pi}$. The line bundle $\lambda(\Pi)$ extends to an ample fractional line bundle $\lambda^*(\Pi)$ on $(\calD_M/\Pi)^*$ and the sections of $m\lambda^*(\Pi)$ are exactly the weight $m$ $\Pi$-automorphic forms. 

We briefly describe the strategy for computing Borcherds' relations which has been used for example in \cite{CML_ij}, \cite{CMJL_quartic} and \cite{LOG_quartick3} (see also \cite{Kondo_kodaira2}, \cite{GHS_kodaira} and \cite{TV-A_kodaira}). The idea is to choose a primitive embedding of $T$ into the even unimodular lattice $II_{26,2} (\cong U^{\oplus 2} \oplus E_8^{\oplus 3}$) of signature $(26,2)$ and study the quasi-pullback $\Phi_T$ (see \cite[\S 3.3.1]{LOG_quartick3}) of Borcherds' automorphic form $\Phi_{12}$ defined in \cite{Borcherds_autoform}. Let $T^{\perp} = T_{II_{26,2}}^{\perp}$ be the orthogonal complement of $T$ in $II_{26,2}$. For a lattice $L$ let us denote the set of roots by $R(L)$. $\Phi_T$ is an automorphic form on $\calD_M$ for $O^*(T)$ of weight $12+|R(T^{\perp})|/2$. The divisor of $\Phi_T$ is supported on the union of the hyperplanes $\delta^{\perp} \cap \calD_M$ where $\delta \in R(II_{26,2}) \backslash R(T^{\perp})$: 
\begin{equation} \label{pre Brocherds}
\mathrm{div}(\Phi_T) = \sum_{\pm \delta \in R(II_{26,2}) \backslash R(T^{\perp})} (\delta^{\perp} \cap \calD_M).
\end{equation} 
Note that $\delta^{\perp} \cap \calD_M \neq \emptyset$ if and only if the lattice $\langle \delta, T^{\perp}\rangle$ spanned by $\delta$ and $T^{\perp}$ is positive definite. Given such a $\delta \in R(II_{26,2}) \backslash R(T^{\perp})$, we let $\nu(\delta)$ be a generator of $(\bQ\delta \oplus \bQ T^{\perp}) \cap T$ (which is unique up to a choice of the sign). As discussed in \cite[Rmk 3.3.3]{LOG_quartick3}, the right hand side of Equation (\ref{pre Brocherds}) is a finite sum of the hyperplane arrangements $\calH_{\nu(\delta)}(O^*(T))$. The coefficient of $\calH_{\nu(\delta)}(O^*(T))$ is $\frac12(|R(\mathrm{Sat}\langle \delta, T^{\perp} \rangle)|-|R(T^{\perp})|)$ where $\mathrm{Sat}$ means taking the saturation in $II_{26,2}$. Equation (\ref{pre Brocherds}) descends to a relation between the Hodge bundle $\lambda(O^*(T))$ on $\calD_M/O^*(T)$ and certain Heegner divisors. By pushing forward via the natural morphism $\calD_M/O^*(T) \rightarrow \calD_M/O^+(T)$, we obtain a Borcherds' relation on $\calD_M/O^+(T)$.  

The strategy is carried out as follows (using the above notations). We first choose a primitive embedding of $T \cong U^{\oplus 2} \oplus D_4^{\oplus 3}$ into $II_{26,2} \cong U^{\oplus 2} \oplus E_8^{\oplus 3}$. 
\begin{lemma} \label{embedTII}
There exists a primitive embedding $T \hookrightarrow II_{26,2}$ with orthogonal complement $T^{\perp}$ isomorphic to $D_4^{\oplus 3}$. 
\end{lemma} 
\begin{proof}
Because $T$ is isomorphic to $U^{\oplus 2} \oplus D_4^{\oplus 3}$ and $II_{26,2}$ is isomorphic to $U^{\oplus 2} \oplus E_8^{\oplus 3}$, it suffices to construct a primitive embedding of $D_4$ into $E_8$. Using Borel-de Siebenthal procedure (see for example \cite[\S 8.2]{Dolgachev_cag}) it is easy to obtain a primitive embedding $D_4 \hookrightarrow D_8 \hookrightarrow E_8$. More explicitly, let $\delta_1, \dots, \delta_8$ be a set of simple roots of $E_8$ (we label them as in \cite[\S VI.4.10]{Bourbaki_lie4-6}). The sublattice generated by $\delta_2$, $\delta_3$, $\delta_4$, $\delta_5$ is isomorphic to $D_4$. The orthogonal complement $\langle \delta_2, \delta_3,\delta_4,\delta_5 \rangle_{E_8}^{\perp}$ is generated by $\delta_7$, $\delta_8$, $\tilde{\delta}=-2\delta_1-3\delta_2-4\delta_3-6\delta_4-5\delta_5-4\delta_6-3\delta_7-2\delta_8$ (the highest root for $E_8$) and $\delta' = 2\delta_1+2\delta_2+3\delta_3+4\delta_4+3\delta_5+2\delta_6+\delta_7$ (coming from the highest root for $D_8$) which is isomorphic to $D_4$.  (For any primitive embedding $D_4 \hookrightarrow E_8$, the orthogonal complement $(D_4)_{E_8}^{\perp}$ is in the genus of $D_4$ and hence must be isomorphic to $D_4$.) The claim then follows. 
\end{proof}

Next we show that the Heegner divisors that show up in the expression of $\mathrm{div}(\Phi_T)$ (more precisely, its descent to $\calD_M/O^*(T)$) are exactly the Heegner divisors $H_0, H_1, \dots, H_{36}$ we introduce before. 

\begin{lemma} \label{classifynudelta}
Choose an embedding of $T$ into the unimodular lattice $II_{26,2}$ as in \lemmaref{embedTII}. Suppose $\delta \in R(II_{26,2}) \backslash R(T^{\perp})$ is such a root that $\langle \delta, T^{\perp} \rangle$ is positive definite (in other words, $\delta^{\perp} \cap \calD_M \neq \emptyset$). Let $\nu(\delta)$ be a generator of $(\bQ\delta \oplus \bQ T^{\perp}) \cap T$. Then one of the following holds (compare \lemmaref{typeab}):
\begin{enumerate}
\item $\nu(\delta)^2=2$ and $\mathrm{div}(\nu(\delta))=1$;
\item $\nu(\delta)^2=4$ and $\mathrm{div}(\nu(\delta))=2$.
\end{enumerate}
\end{lemma}
\begin{proof}
The proof is analogous to that of \cite{LOG_quartick3} Proposition 3.3.10. Let $m$ be the minimal positive integer such that $m \delta \in T \oplus T^{\perp}$. Write $$m\delta = v+w$$ where $0 \neq v \in T$ and $w \in T^{\perp}$. Because the embedding $T \hookrightarrow II_{26,2}$ is defined for every piece of $D_4$ contained in $T \cong U^{\oplus 2} \oplus D_4^{\oplus 3}$, we have $m \in \{1,2,4\}$ (N.B. $[E_8: D_4 \oplus D_4] =4$). Note also that $v \in \langle \nu(\delta) \rangle$ and $\nu(\delta) = \pm v$ if and only if $v$ is primitive. Because $\langle \delta, T^{\perp} \rangle$ is positive definite, we get $v^2 > 0$ and $w^2 \geq 0$. If $w^2=0$, then $\nu(\delta) = \pm \delta$ and Case (1) holds. Since $2m^2 = v^2 + w^2$, one of the following holds:
\begin{enumerate}
\item[(i)] $m=1$, $v^2=2$, and $\nu(\delta) = \pm v$; 
\item[(ii)] $m=2$, $v^2 \in \{2,4,6\}$, $\nu(\delta) = \pm v$, and $\mathrm{div}(v)$ is either $2$ or $4$ in $T$;
\item[(iii)] $m=4$, $v^2 \in \{2,4,6, \dots, 30\}$, $\nu(\delta) = \pm v$, and $\mathrm{div}(v)$ equals $4$ in $T$;
\item[(iv)] $m=4$, $v^2 \in \{8, 16, 18, 24\}$, and $v$ is not primitive.
\end{enumerate}  
Let us do a case by case analysis.
\begin{enumerate}
\item[(i)] Suppose that (i) holds, then we have Case (1). 
\item[(ii)] Suppose that (ii) holds. Because the discriminant quadratic form $q_T$ takes values in integers, the only possibility is $v^2=4$, $\nu(\delta) = \pm v$, and $\mathrm{div}(v)$ equals $2$ in $T$. This is Case (2).
\item[(iii)] We claim (iii) can not happen. It contradicts the fact that $A_T \cong (\bZ/2\bZ)^6$. 
\item[(iv)] We claim (iv) can not happen. Suppose (iv) holds. Let us observe that $w$ must be primitive. When $v^2=18$, this is clear. When $v^2 \in \{8,16,24\}$, $w^2 \in \{8,16, 24\}$. If $w$ is not primitive, then one has $v=2u$ and $w=2z$ for $u \in T$ and $z \in T^{\perp}$. Thus, $2\nu(\delta) = u+z$ which contradicts our assumption that $m$ is minimal. But then $\mathrm{div}(w)=4$ in $T^{\perp}$ which is impossible ($A_{T^{\perp}} \cong A_T \cong (\bZ/2\bZ)^6$ does not have 4-torsion elements).      
\end{enumerate}  
\end{proof}

Now let us compute Borcherds' relations. 
\begin{proposition} \label{BorcherdsO+}
In the $\bQ$-Picard group $\Pic(\calD_M/O^+(T))_{\bQ}$ we have $$\lambda(O^+(T)) \sim H_n + 2 H_t$$ where $D \sim D'$ means $D = c D'$ for some $c \in \bQ^*$.
\end{proposition}
\begin{proof}
Let us embed $T$ into $II_{26,2}$ as in \lemmaref{embedTII}. Let $\Phi_T$ be the quasi-pullback of Borcherds' automorphic form $\Phi_{12}$ (see \cite[\S 3.3.1]{LOG_quartick3}). Then $\Phi_T$ is automorphic form on $\calD_M$ for $O^*(T)$. The weight of $\Phi_T$ is $12+|R(T^{\perp})|/2 = 12+|R(D_4^{\oplus 3})|/2 = 12+72/2=48$ (recall that $|R(D_k)|=2k(k-1)$). Set $\delta \in R(II_{26,2}) \backslash R(T^{\perp})$ to be a root with $\langle \delta, T^{\perp} \rangle$ positive definite and denote a generator of $(\bQ\delta \oplus \bQ T^{\perp}) \cap T$ by $\nu(\delta)$. Then $\Phi_T$ vanishes on a union of hyperplane arrangements $\calH_{\nu(\delta)}(O^*(T))$ (cf. \cite[Rmk. 3.3.3]{LOG_quartick3}). By \lemmaref{classifynudelta}, the vectors $\nu(\delta) \in T$ defining these hyperplane arrangements are exactly the vectors we consider in \lemmaref{typeab}. Type (1) vectors in \lemmaref{typeab} form a single $O^*(T)$-orbit and there are $36$ orbits for Type (2) vectors. For every orbit it is easy to find a vector which can be realized as $\nu(\delta)$ for some $\delta \in R(II_{26,2}) \backslash R(T^{\perp})$. The vanishing order of $\Phi_T$ along $\calH_{\nu(\delta)}(O^*(T))$ equals $\frac12(|R(\mathrm{Sat}\langle \delta, T^{\perp} \rangle)|-|R(T^{\perp})|)$. If $\nu(\delta)$ is of Type (1), then the lattice $\langle \delta, T^{\perp} \rangle$ is saturated and isomorphic to $A_1 \oplus D_4^{\oplus 3}$. If $\nu(\delta)$ is of Type (2), then the saturation of the lattice $\langle \delta, T^{\perp} \rangle$ is $D_5 \oplus D_4^{\oplus 2}$. Note also that $H_0$ is reflective. Putting everything together, we get an expression (in $\Pic(\calD_M/O^*(T))_{\bQ}$) for the descent of $\mathrm{div}(\Phi_T)$ from $\calD_M$ to $\calD_M/O^*(T)$: $48 \lambda(O^*(T)) = \frac12 H_0 + 8(H_1 + \dots + H_{36})$. Now let us push forward this relation via $\calD_M/O^*(T) \rightarrow \calD_M/O^+(T)$. This means one needs to divide the Borcherds' relation by the ramification order for the Heegner divisors $H_t$ which is $8$ (by \propositionref{geoHeegner} a generic point of $H_t$ corresponds to a cubic fourfold with a pair of conjugate $A_1$ singularities and hence the local monodromy group has order $8$: the product of Weyl groups for the two $A_1$ singularities and a involution interchanging the two singularities). Thus we obtain $\lambda(O^+(T)) \sim H_n + 2 H_t$.  
\end{proof}

\section{Extending the period map}
We compactify the period map $\calP_0: \calM_0 \rightarrow \calD_M/O^+(T)$ defined in Section \ref{pairs torelli}. Specifically, we show that a certain GIT compactification of the moduli $\calM_0$ of the pairs $(X,H)$ is isomorphic to the Baily-Borel compactification of the locally symmetric domain $\calD_M/O^+(T)$. 

The natural parameter space for cubic pairs $(X,H)$ consisting of a cubic threefold $X$ and a hyperplane $H$ is $$\bfP := \bP H^0(\bP^4, \calO(3)) \times \bP H^0(\bP^4, \calO(1)) \cong \bP^{34} \times \bP^{4}$$ on which the group $G := \SL(5, \bC)$ acts diagonally. Thus we consider the GIT quotient $$\bfP \git_{\calL} G = \Proj \bigoplus_{m \geq 0}H^0(\bfP, \calL^{\otimes m})^G$$ where $\calL$ is an ample $G$-linearized line bundle. The dependence of the GIT quotient on the choice of an ample $G$-linearized line bundle was studied by Thaddeus \cite{Thaddeus_vgit} and Dolgachev and Hu \cite{DH_vgit}. Note that $\Pic^{G}(\bfP) \cong \Pic(\bfP) \cong \bZ \times \bZ$. Following \cite[Def. 2.2]{Laza_n16}, an ample $G$-linearized line bundle $\calL$ is said to be of slope $t \in \bQ_{\geq 0}$ if $\calL \cong \pi_1^*\calO_{\bP^{34}}(a) \otimes \pi_2^* \calO_{\bP^{4}}(b)$ with $t = \frac{b}{a}$. By \cite{Thaddeus_vgit} and \cite{DH_vgit} the quotient $\bfP \git_{\calL} G$ only depends on the slope $t$ of $\calL$. We denote the corresponding GIT by $\calM(t)$ or $\bfP \git_t G$. The lower and upper bounds for $t$ are $0$ and $\frac34$ respectively (i.e. $\calM(t) = \emptyset$ if $t<0$ or $t > \frac34$). Let $(X,H)$ be a $t$-semistable pair. As $t$ increases, $X$ may become more singular but we require better transversality for $X \cap H$. More precisely, we have the following proposition.

\begin{proposition} \label{vgitinterval}
Let $(X,H) \in \bfP$. Then there exists an interval (possibly empty) $[a,b] \subset [0, \frac34]$ such that $(X,H)$ is $t$-semistable if and only if $t \in [a,b] \cap \bQ_{\geq 0}$. Also, if $(X,H)$ is $t$-stable for some $t$ then it is $t$-stable for all $t \in (a,b) \cap \bQ_{\geq 0}$. Furthermore,
\begin{enumerate}
\item $a=0$ if and only if $X$ is a GIT semistable cubic threefold (which has been described in \cite{Allcock_cubicgit});
\item $b=\frac34$ if and only if $X \cap H$ is a GIT semistable cubic surface (see for example \cite[\S 7.2]{Mukai_git}). 
\end{enumerate}
\end{proposition}
\begin{proof}
See \cite[Thm. 2.4]{Laza_n16} or \cite[Cor. 3.3, Lem. 4.1]{GMG_vgit}.
\end{proof}

By \propositionref{vgitinterval} the cubic pairs $(X,H)$ with $X$ at worst nodal (i.e. admitting nodes, or equivalently, $A_1$ hypersurface singularities, cf. \cite[\S15]{AGZV1}) and $H$ at worst simply tangent to $X$ (that is, $X \cap H$ admits one $A_1$ singularity; in particular, $H$ does not pass through any singular point of $X$) are stable for any $t \in (0, \frac34)$. Let $\calM \subset \calM(t)$ ($t \in (0, \frac34)$) be the moduli space of cubic pairs $(X,H)$ consisting of a cubic threefold $X$ with at worst nodal singularities and a hyperplane $H$ which is at worst simply tangent to $X$. Note that $\calM$ is a geometric quotient and a quasi-projective variety. 

For $(X,H) \in \calM$ we also consider the cubic fourfold $Y$ defined in (\ref{Y eqn}). For projectively equivalent pairs $(X,H)$ the corresponding cubic fourfolds $Y$ are also projectively equivalent. When $X$ has a node and $H$ is in general position (i.e. $H$ avoids the node and intersects $X$ transversely), the cubic fourfold $Y$ also has a node. When $X$ is smooth and $H$ is simply tangent to $X$, the cubic fourfold $Y$ admits a pair of $A_1$ singularities which are conjugate to each other with respect to the involution $\sigma: [y_0, \dots, y_4,y_5] \mapsto [y_0, \dots, y_4, -y_5]$. 

We recall the results of Looijenga \cite{Looijenga_cubic} and the first author \cite{Laza_cubicgit, Laza_cubic} on the image of the period map $\calC_0 \rightarrow \calD/\Gamma^+$ for smooth cubic fourfolds. Notations as in Section \ref{pairs torelli}. For a saturated rank $2$ sublattice $K \subset \Lambda$, we consider the hyperplane $\calD_K :=\{\omega \in \calD \,|\, \omega \perp K\}$. Geometrically, $\calD_K$ corresponds to certain special cubic fourfolds (see \cite[\S3.1, \S4]{Hassett_cubic}). The hyperplane $\calD_K$ is said to be of discriminant $d = \det(K)$. The hyperplanes $\calD_K$ of a given discriminant $d$ form an arithmetic arrangement with respect to $\Gamma^+$ (cf. \cite[\S 3.2]{Hassett_cubic}). Let $\calH_{\infty} \subset \calD$ (resp. $\calH_{\Delta} \subset \calD$) be the arrangement of hyperplanes of discriminant $2$ (resp. $6$). Note that $\calH_\infty/\Gamma^+$ and $\calH_\Delta/\Gamma^+$ are irreducible hypersurfaces in $\calD/\Gamma^+$. According to \cite[Thm. 1.1]{Laza_cubic}, the image of the period map $\calC_0 \rightarrow \calD/\Gamma^+$ is the complement of the hyperplane arrangement $\calH_\infty \cup \calH_\Delta$. 

A generic point of $\calH_\infty/\Gamma^+$ corresponds to a certain determinantal cubic fourfold (see \cite[\S 4.4]{Hassett_cubic}). Moreover, we have the following lemma. 
\begin{lemma} \label{H_infty empty}
The intersection of the subdomain $\calD_M$ and the arrangement of hyperplanes $\calH_{\infty}$ of discriminant $2$ is empty. 
\end{lemma}
\begin{proof}
Let $K \subset \Lambda$ be a saturated rank $2$ sublattice with discriminant $2$. Denote the corresponding hyperplane in $\calD$ by $\calD_K$ which is a member of $\calH_{\infty}$. By \cite[\S4.4]{Hassett_cubic} 
$K$ is generated by $h^2$ and an element $x$ satisfying $x \cdot h^2 =1$ and $x^2=1$. Suppose that $\calD_M \cap \calD_K \neq \emptyset$ in $\calD$. Then the lattice generated by $M$ and $x$ is positive definite (Hodge-Riemann bilinear relations). Recall that $M \subset \Lambda$ is generated by $f_0, \dots, f_6$ with the intersection form given by the Gram matrix in \lemmaref{N gram}. A direct computation shows that $x \cdot f_0 = 0$, $1$ or $2$ and $x \cdot f_i = 0$ or $1$ ($1 \leq i \leq 6$). (For example, the Gram matrix of the lattice generated by $h^2$, $x$ and $f_i$ is as follows.
$$
\begin{pmatrix} 
3 & 1 & 1   \\
1 & 1 & x \cdot f_i   \\
1 & x \cdot f_i & 3   \\
\end{pmatrix}
$$
Because the matrix is positive definite, one has $x \cdot f_i = 0$ or $1$.) Note that $h^2 = 3f_0-f_1-\dots-f_6$ and $h^2 \cdot x =1$. There are two possibilities: either $x \cdot f_0 = 1$ (and $x \cdot f_i=1$ for two $f_i$'s) or $x \cdot f_0 = 2$ (and $x \cdot f_i=1$ for five $f_i$'s). But in both cases the Gram matrix of the lattice generated by $x$ and $M$ 
has determinant $0$ which is a contradiction. 
\end{proof}

Denote the moduli of cubic fourfolds with at worst simple singularities by $\calC$ (cf. \cite[Thm. 1.1]{Laza_cubicgit}). By \cite[Thm. 1.1]{Laza_cubic} the period map extends to an isomorphism $\calC \rightarrow (\calD \, \backslash \, \calH_{\infty})/\Gamma^+$ with image the complement of $\calH_\infty$. In particular, $\calC \backslash \calC_0$ is mapped to $\calH_{\Delta}/\Gamma^+$. A generic point of $\calH_\Delta/\Gamma^+$ corresponds to a cubic fourfold with $A_1$ singularities (see also \cite[\S 4]{Voisin_cubic} and \cite[\S 4.2]{Hassett_cubic}). 

We analyze how $\calH_{\Delta}$ intersects with $\calD_M$. By the previous paragraph, generically $\calH_{\Delta}$ parameterizes cubic fourfolds admitting $A_1$ singularities. Let $Y$ be a cubic fourfold coming from a cubic pair $(X,H)$. If $Y$ is singular then either $X$ is singular or $H$ is tangent to $X$ (see the proof of \lemmaref{Y smooth}). Recall that $\calM$ parameterizes cubic pairs $(X,H)$ with $X$ at worst nodal and $H$ at worst simply tangent to $X$. There are two natural geometric divisors of $\calM$: cubic pairs with $X$ admitting at least one node and $H$ in general position (call the closure $\Sigma$) and cubic pairs with $X$ smooth and $H$ simply tangent to $X$ (call the closure $V$). Now we extend the period map $\calP_0$ from $\calM_0$ to $\calM$ and match $\Sigma$ (resp. $V$) with the nodal Heegner divisor $H_n$ (resp. the tangential Heegner divisor $H_t$) introduced in Section \ref{Heegner}. As a result, the intersection of $\calH_{\Delta}$ and $\calD_M$ produces the Heegner divisors $H_n$ and $H_t$. 

\begin{proposition} \label{extend P0}
The period map $\calP_0: \calM_0 \rightarrow \calD_M/O^+(T)$ extends to a morphism $\calP: \calM \rightarrow \calD_M/O^+(T)$. 
\end{proposition}
\begin{proof}
We apply the removable singularity theorem (see for example \cite[p. 41]{Griffiths_tag}) and \cite[Proposition 3.2]{Laza_cubic}. Let $o \in \calM \,\backslash\, \calM_0$ correspond to a pair $(X_0,H_0)$. The corresponding cubic fourfold $Y_0$ has at worst $A_1$ singularities. The statement is analytically local at $o$, and stable by finite base change. Since $\calM$ is a geometric quotient we can assume (after shrinking and a possible finite cover) that a neighborhood of $o$ in $\calM$ is a $14$-dimensional ball $B$. We can further assume that there exists a family of cubic pairs over $B$ (and hence there is a family of cubic fourfolds $\calY \rightarrow B$ with at worst $A_1$ singularities and fiber $Y_0$ over $o$). Let $\Omega$ be the discriminant hypersurface. Over $B \, \backslash \, \Omega$ the family $\calY$ gives a variation of Hodge structure on $\Lambda_0 \cong H^4_{\mathrm{prim}}(Y, \bZ)$ (which defines the period map $B \, \backslash \, \Omega \rightarrow \calD/\Gamma^+$) and a variation of Hodge structure on the transcendental part $T$ (which corresponds to $\calP_0: B \, \backslash \, \Omega \rightarrow \calD_M/O^+(T)$). By the removable singularity theorem, the proposition is equivalent to the monodromy representation $\pi_1(B \,\backslash \, \Omega, t) \rightarrow \Aut(T)$ (where $t \in B \, \backslash \, \Omega$) having finite image. By \cite[Proposition 3.2]{Laza_cubic} the period map $B \, \backslash \, \Omega \rightarrow \calD/\Gamma^+$ extends across the point $o$. It follows that the local monodromy for the variation of Hodge structure on $\Lambda_0$ around $Y_0$ is finite. Consider the restriction of the monodromy group to $T$. We conclude that $\calP_0: \calM_0 \rightarrow \calD_M/O^+(T)$ has finite local monodromy.
\end{proof}

We compare the geometric divisor $\Sigma$ (resp. $V$) with the nodal Heegner divisor $H_n$ (resp. the tangential Heegner divisor $H_t$). 
\begin{lemma} \label{geoHeegner}
The generic point of the Heegner divisor $H_n$ (resp. $H_t$) corresponds to a cubic pair $(X,H)$ with $X$ admitting an $A_1$ singularity and $H$ in general position (resp. a cubic pair $(X,H)$ with $X$ smooth and $H$ simply tangent to $X$) via the extended period map $\calP$.
\end{lemma}
\begin{proof}
If $X$ has an $A_1$ singularity and $H$ is a general hyperplane, then the associated cubic fourfold $Y$ has a single $A_1$ singularity. The limiting mixed Hodge structures of nodal cubic fourfolds have been studied in \cite[\S 4]{Voisin_cubic} and \cite[\S 4.2]{Hassett_cubic}. Specifically, we project $Y$ from the node. The surface parameterizing the lines of $Y$ through the node is a $K3$ surface $S$ (which is the complete intersection of a quadratic and a cubic in $\bP^4$). The desingularization of $Y$ is isomorphic to the blowup of $\bP^4$ along $S$. This (together with the Clemens-Schmid sequence) induces an embedding of $H^2(S)(-1)$ into the limiting mixed Hodge structure $H_{\mathrm{lim}}^4(Y)$. The orthogonal complement is generated by a vector $v_n \in T$ with $v_n^2=2$. Thus, the corresponding period point belongs to $H_n$.

Suppose $X$ is smooth and $H$ is simply tangent to $X$. The associated cubic fourfold $Y$ has a pair of $A_1$ singularities $p_1$ and $p_2$. Note that $p_1$ and $p_2$ are conjugate to each other under the involution $\sigma: [y_0, \dots, y_4,y_5] \mapsto [y_0, \dots, y_4, -y_5]$. Let us project $Y$ from $p_1$ and from $p_2$ to a common complementary hyperplane $(y_5=0)$ in $\bP^5$. As in the previous paragraph (see also \cite[Prop. 3.8]{Laza_cubic}), we get a nodal $K3$ surface for $p_1$ and a nodal $K3$ surface for $p_2$. It is not difficult to prove that these $K3$ surfaces coincide. We blow up the node of the $K3$ surface (call the smooth $K3$ surface $S$ and the class of exceptional curve $e$), embed $H^2(S)(-1)$ into $H^4_{\mathrm{lim}}(Y)$, and take the orthogonal complement (call the generator $v$). Then $v_t := e+v$ is a vector of $T$ with $v_t^2=4$ and $\mathrm{div}(v_t)=2$. The lemma then follows. (We also observe the following connection between $(X,H)$ and $v_t$. Note that the cubic surface $X \cap H$ has one $A_1$ singularity. It is classically known that $X \cap H$ contains $21$ lines (there are $6$ lines through the node which are limits of $6$ pairs of lines on a smooth cubic surface). In other words, a (marked) cubic surface containing one node determines a double-six (cf. 
\cite[\S 9.1]{Dolgachev_cag}) of lines. By \cite[Lem. 9.1.2]{Dolgachev_cag} a double-six corresponds to a pair of opposite roots $\pm \alpha$ of $E_6$. Define $q: E_6/2E_6 \rightarrow \bZ/2\bZ$ by $q(-) \equiv \frac12(-,-)_{E_6} \pmod{2\bZ}$ (cf. \lemmaref{qTe6}). The roots $\pm \alpha$ gives an element $\bar{\alpha} \in E_6/2E_6$ with $q(\bar{\alpha}) \equiv 1 \pmod{2\bZ}$. Recall that we have $(E_6/2E_6, q) \cong (A_T, q_T)$ which also induces $W(E_6) \cong O(q_T)$. The $W(E_6)$-orbit of $\bar{\alpha}$ coincides with the $O(q_T)$-orbit of $\hat{v}_t=v_t/\mathrm{div}(v_t)$.)
\end{proof}

Now let us show that the GIT compactification $\calM(\frac13)$ is isomorphic to the Baily-Borel compactification of $\calD_M/O^+(T)$. We follow the general framework developed by Looijenga \cite{Looijenga_ball, Looijenga_typeiv}. Note that we get the Baily-Borel compactification because there is no hyperplane arrangement missing in the image of the extended period map $\calP$. 
 
\begin{theorem} \label{gitbb}
The period map $\calP_0: \calM_0 \rightarrow \calD_M/O^+(T)$ extends to an isomorphism $\calM(\frac13) \cong (\calD_M/O^+(T))^*$ where $(\calD_M/O^+(T))^*$ denotes the Baily-Borel compactification of $\calD_M/O^+(T)$.
\end{theorem}
\begin{proof}
Consider the open subset $\calU'$ of $\bfP = \bP H^0(\bP^4, \calO(3)) \times \bP H^0(\bP^4, \calO(1))$ parameterizing cubic pairs $(X,H)$ with $X$ at worst nodal and $H$ at worst simply tangent to $X$. Clearly, $\calU' \subset \bfP$ is invariant under the action of $G = \SL(5, \bC)$ and $\calM = \calU'/G$. Moreover, the complement $\bfP \backslash \calU'$ has codimension higher than $1$ (N.B. $\calM(t)$ and $\calM(t')$ only differ in codimension $2$ for $t \neq t'$). Using the extended period map $\calP: \calM \rightarrow \calD_M/O^+(T)$ we identify a $G$-invariant open subset $\calU \subset \calU'$ with an $O^+(T)$-invariant open subset $\calD_M^o \subset \calD_M$. Specifically, a cubic pair $(X,H) \in \calU'$ belongs to $\calU$ if the associated cubic fourfold $Y$ has exactly one Eckardt point and its $G$-orbit corresponds to a smooth point of $\calU'/G$. Again $\bfP \backslash \calU$ has codimension bigger than $1$ (cf. \cite[Thm. 2.10 and \S6]{CC_star}). The restriction of the extended period map $\calP$ to $\calU/G$ is an isomorphism onto its image $\calD_M^o/O^+(T)$ ($\calP|_{\calU/G}$ is injective by \theoremref{P0 injective} and \cite[Thm. 1.1]{Laza_cubic}, and it is a local isomorphism by the same argument in the proof of \propositionref{dP0 injective}). Next we argue that the codimension of the complement of $\calD_M^o$ in $\calD_M$ is larger than $1$. According to \cite[Thm. 1.1]{Laza_cubic} (see also \propositionref{Hodge Eckardt}), a period point $\omega \in \calD_M$ which is not contained in the arrangement of hyperplanes $\calH_{\Delta}$ or $\calH_{\infty}$ corresponds to a cubic pair $(X,H)$ where $X$ is smooth and $H$ is transverse to $X$. \lemmaref{H_infty empty} tells us that $\calH_{\infty}$ does not meet $\calD_M$. A generic period point $\omega_{\Delta} \in \calH_{\Delta}$ corresponds to a cubic fourfold $Y$ with $A_1$ singularities (cf. \cite[\S 4]{Voisin_cubic}, \cite[\S 4.2]{Hassett_cubic} and \cite[Thm. 1.1]{Laza_cubic}). Suppose $Y$ comes from a cubic pair $(X,H) \in \calU$. Then either $X$ is nodal or $H$ is simply tangent to $X$. Thus, $\omega_{\Delta}$ corresponds to a generic point of the geometric divisor $\Sigma$ or $V$. Now it suffices to show that the extended period map $\calP$ preserves the polarizations. We have computed the Borcherds' relation between the Hodge bundle $\lambda(O^+(T))$ on $\calD_M/O^+(T)$ and the Heegner divisors $H_n$ and $H_t$ in Section \ref{Heegner}. From \propositionref{BorcherdsO+} we get $\lambda(O^+(T)) \sim H_n + 2 H_t$. We have also matched the geometric divisor $\Sigma$ (resp. $V$) with the Heegner divisor $H_n$ (resp. $H_t$) in \propositionref{geoHeegner}. Write $\calO(a,b):=p_1^*\calO(a) \otimes p_2^*\calO(b)$ (which is a $G$-linearized line bundle on $\bfP$). From \cite[Rmk. 1.2, Thm. 1.3]{Benoist_singularci} we deduce that $\Sigma = \calO(80, 0)$ and $V=\calO(32,24)$. The line bundle $\calL$ corresponding to $\lambda(O^+(T))$ is hence $\calO(144,48)$ which has slope $\frac13$. This essentially completes the proof of the theorem. Indeed, recall that $\calM(\frac13) = \Proj \bigoplus_{m \geq 0}H^0(\bfP, \calL^{\otimes m})^G$ and $(\calD_M/O^+(T))^* = \Proj \bigoplus_{m \geq 0}H^0(\calD_M, \lambda^{\otimes m})^{O^+(T)}$ (where $\lambda=\calO_{\calD_M}(-1)$ is the natural automorphic bundle over $\calD_M$). The polarized isomorphism $\calP|_{\calU/G}$ induces an isomorphism between $H^0(\calU, \calL^{\otimes m})^G$ and $H^0(\calD_M^o, \lambda^{\otimes m})^{O^+(T)}$ (see also the proof of \cite[Thm. 7.6]{Looijenga_typeiv}). Because $\calU \subset \bfP$ has codimension higher than $1$, we have $H^0(\calU, \calL^{\otimes m})^G \cong H^0(\mathbf{P}, \calL^{\otimes m})^G$. Similarly, $H^0(\calD_M^o, \lambda^{\otimes m})^{O^+(T)} \cong H^0(\calD_M, \lambda^{\otimes m})^{O^+(T)}$.
\end{proof}

\begin{remark}
The stability of pairs $(X,H)$ for slope $t = \frac13$ is closely related to the stability of the associated cubic fourfold Y. In particular, if $Y$ has at worst simple singularities\footnote{We recall that simple singularities is the same thing as A-D-E hypersurface singularities, i.e. suspensions (in the fourfold case, double-suspensions) of the eponymous surface singularities. The relevance of these singularities in the context of period maps is that in even dimensions they lead to finite local monodromy, and thus extensions of period maps. } then $(X,H)$ is stable for $t=\frac13$. \end{remark}

\appendix

\section{Algebraic varieties of $K3$ type} \label{wps}

\par The study of families of algebraic varieties whose period map takes values in the quotient of a Hermitian symmetric domain of type IV has a long and rich history. In this appendix we discuss two such classes of varieties. The difference is akin to the difference between Kunev surfaces (which have correct Hodge numbers $h^{2,0}=1$ but the holomorphic $2$-forms vanish along curves) v.s.~$K3$ surfaces (which have nondegenerate holomorphic $2$-forms). The emphasis will be on weighted hypersurfaces (especially weighted fourfolds). In particular, we classify in \theoremref{quasi-k3-4fold} families of weighted fourfolds satisfying:
\begin{enumerate}
\item A general member is a quasi-$K3$;
\item The family contains a Fermat hypersurface.  
\end{enumerate}

\par Let $\mathbb W$ be a well formed weighted projective space and 
$s(\mathbb W)$ be the sum of the weights of $\mathbb W$. Let $Z$ be a quasi-smooth closed subvariety of $\mathbb W$. Then, $H^k(Z,\mathbb Q)$ admits a pure Hodge structure of weight $k$. Suppose $Z$ has complex dimension $2n$. 
\begin{definition}
\leavevmode
\begin{itemize}
\item We say that $Z$ is a \emph{numerical $K3$} if $H^{2n}(Z,\mathbb Q)$ is a Hodge structure of level 2 with $h^{n+1,n-1}(Z)=1$. 
\item In the case where $Z=(f=0)$ is a quasi-smooth hypersurface, we say that $Z$ is a \emph{quasi-$K3$} if 
           \begin{equation}
           \frac{1}{2}(\dim_{\mathbb C} Z)\text{\rm deg}(f) 
            = s(\mathbb W).    
           \label{quasi-K3}
           \end{equation}
\end{itemize}
\end{definition}

\begin{lemma} 
Every quasi-$K3$ is a numerical $K3$.
\end{lemma}
\begin{proof} 
We begin by recalling the following general description of the cohomology of a quasi-smooth hypersurface in a weighted projective space~\cite{Dolgachev_wps}:

\par Let $R=\mathbb C[z_0,\dots,z_m]$ where $z_j$ has weight $w_j$, and 
$\mathbb U$ be the associated weighted projective space.  Let 
$d\text{\rm Vol} = dz_0\wedge\cdots\wedge dz_m$ be the projective 
volume form and $E=\sum_{j=0}^m\, w_jz_j\frac{\partial}{\partial z_j}$ be 
the Euler vector field.  Then, $\Omega = i(E)d\text{\rm Vol}$ is 
a homogeneous differential form of degree 
$s(\mathbb U)=\sum_{j=0}^m\, w_j$. If 
$V = (g=0) \subset \mathbb U$ is a quasi-smooth hypersurface, then 
$F^{m-q}H^{m-1}_{\mathrm{prim}}(V)$ is spanned by the Poincar\'e residues 
of the algebraic differential forms 
$\Omega(A) = (A\Omega)/(g^q)$ of homogeneous degree $0$ (where $A$ is 
a homogeneous polynomial on $\mathbb U$).  The residue of $\Omega(A)$ belongs to $F^{m+1-q} H^{m-1}_{\mathrm{prim}}(V)$ if and only if $A$
belongs to the Jacobian ideal $J(g)$ of $g$. In particular, this description of the Hodge filtration of the primitive cohomology of $V$ implies that:
\begin{equation}
         h^{m-1-j,j}_{\mathrm{prim}}(V) 
         = \dim (R/J(g))_{(j+1)\text{\rm deg}(g)-s(\mathbb U)}
         \label{hodge-numbers}
\end{equation}
where $(R/J(g))_{\ell}$ is the subspace of homogeneous forms of degree 
$\ell$ in $R/J(g)$.

\par If $Z\subset\mathbb W$ is a quasi-smooth hypersurface of dimension $2n$ defined by the vanishing of a polynomial $f$ of degree $d$ then application of \eqref{hodge-numbers} with $m=2n+1$ and $s(\mathbb W) = nd$ by \eqref{quasi-K3} implies that 
$
 h^{2n-j,j}_{\mathrm{prim}}(Z) = \dim (R/J(f))_{(j+1-n)d}
$
which is zero for $j<n-1$, and equals to $\dim (R/J(f))_0 = 1$ for $j=n-1$ (N.B. $h^{n+1,n-1}(Z) = h^{n+1,n-1}_{\mathrm{prim}}(Z)$).
\end{proof}

\begin{remark} 
The Fermat surface $Z$ of degree 15 in $\mathbb P(1,3,5,5)$ gives an example of a quasi-smooth hypersurface which is numerically of $K3$ type but not a quasi-$K3$. Indeed, by Equation \eqref{hodge-numbers}, 
$
       h^{2,0}(Z) = \dim (R/J)_{15-14} = 1
$.
\end{remark}

\begin{remark}
Suppose $Z = (f=0)$ is a numerical $K3$ weighted hypersurface which is not a quasi-$K3$. Let $\dim_{\mathbb C} Z = 2n$. Then $H^{n+1, n-1}(Z)$ is generated by the Poincar\'e residue of an algebraic differential form $\Omega(A) = (A \Omega)/(f^{n})$ where $A$ has positive degree. A weighted version of the ``\v{C}echist" residue formula in \cite[\S3.B]{CG_infvhs} allows us to show that the residue of $\Omega(A)$ vanishes along a divisor.
\end{remark}

\begin{remark} 
Each surface appearing on Reid's list (see for instance \cite{Reid_weightedk3} or \cite{Yonemura_wtk3}) of $95$ weighted $K3$ hypersurfaces is a quasi-$K3$.
\end{remark}

\par If $Z=(\sum_j\, z_j^{n_j} = 0)$ is a quasi-$K3$ Fermat (all $n_j>1$) then dividing both sides of \eqref{quasi-K3} by the degree of $f$ implies that
\begin{equation}
        \sum_j\, \frac{1}{n_j} = \frac12 \dim_{\mathbb C}Z    \label{recip-weight}.
\end{equation}
In the case where $Z$ is a surface, this leads to the following $14$ partitions of $1$ into a sum of $4$ unit fractions:
\begin{equation}
\aligned
        1 &= \frac12 + \frac13 + \frac17 + \frac{1}{42}  
           =\, \frac12 + \frac13 + \,\frac18\, + \frac{1}{24}  \\
          &= \frac12 + \frac13 + \frac19 + \frac{1}{18}  
           = \frac12 + \frac13 + \frac{1}{10} + \frac{1}{15} \\
          &= \frac12 + \frac13 + \frac{1}{12} + \frac{1}{12} 
           = \frac12 + \frac14 + \frac15 + \frac{1}{20}  \\
          &= \frac12 + \frac14 + \frac16 + \frac{1}{12} 
           = \frac12 + \frac14 + \frac18 + \frac18 \\
          &= \frac12 + \frac15 + \frac15 + \frac{1}{10} 
           = \frac12 + \frac16 + \frac16 + \frac16  \\
          &= \frac13 + \frac13 + \frac14 + \frac{1}{12} 
           = \frac13 + \frac13 + \frac16 + \frac16 \\
          &= \frac13 + \frac14 + \frac14 + \frac16 
           = \frac14 + \frac14 + \frac14 + \frac14.
\endaligned
\label{partition-4}
\end{equation}

\begin{theorem} \label{quasi-k3-4fold} 
Let $Z=(z_0^{n_0}+z_1^{n_1}+z_2^{n_2}+z_3^{n_3}+z_4^{n_4}+z_5^{n_5}=0)$ be a quasi-$K3$ Fermat fourfold (all $n_j>1$).  Then,
$
      2 = \frac{1}{n_0} + \frac{1}{n_1} + \frac{1}{n_2} + \frac{1}{n_3} + \frac{1}{n_4} + \frac{1}{n_5}
$
is either a sum $\frac12 + \frac12 + \frac1a + \frac1b + \frac1c + \frac1d$, and hence classified by \eqref{partition-4} or a sum $\alpha + \beta$ where $\alpha$ and $\beta$ are one of the following partitions:
\begin{equation}
     1 = \frac13+\frac13+\frac13 = \frac12 + \frac13 + \frac16 = \frac12 + \frac14 + \frac14. \label{partition-3}
\end{equation}
\end{theorem}
\begin{proof} By reordering the variables as necessary, we assume that $1 > \frac{1}{n_0} \geq \dots \geq \frac{1}{n_5}$. Then, either $\frac{1}{n_0} + \frac{1}{n_1} + \frac{1}{n_2} = 1$ or $\frac{1}{n_0} + \frac{1}{n_1} + \frac{1}{n_2}>1$. In the first case, we have a partition $2=\alpha + \beta$ with $\alpha$, $\beta$ from \eqref{partition-3}. In the second case, either $\frac{1}{n_0} = \frac{1}{n_1} = \frac12$ or $\frac{1}{n_0}+\frac{1}{n_1}<1$. If $\frac{1}{n_0} = \frac{1}{n_1}=\frac12$ the remaining fractions give a partition from \eqref{partition-4}. If $\frac{1}{n_0}+\frac{1}{n_1}<1$ then $\frac{1}{n_0}+\frac{1}{n_1} + \frac{1}{n_2}$ equals one of the following sums: $\frac12+\frac13+\frac13 = \frac76$, $\frac12+\frac13+\frac14=\frac{13}{12}$ or $\frac12+\frac13+\frac15 = \frac{31}{30}$. If $\frac{1}{n_0} + \frac{1}{n_1} + \frac{1}{n_2} = \frac76$ then $\frac{1}{n_3} + \frac{1}{n_4} + \frac{1}{n_5} = \frac56$. This forces $\frac{1}{n_3} = \frac13$.  But then $\frac{1}{n_0} + \frac{1}{n_4} + \frac{1}{n_5}=1$. Similar arguments work for the other two cases. 
\end{proof}

\par Altogether, \theoremref{quasi-k3-4fold} produces 17 families of quasi-smooth weighted fourfolds which are summarized in Table \ref{wps-table}. For every family, a general hypersurface is a quasi-$K3$ and there exists a Fermat member.

\begin{table}[h!]
\centering
\begin{tabular}{|c|c|c|c||c|c|c|c|}
\hline
Case & WPS & $\deg$ & $h_{\mathrm{prim}}^{2,2}$ & Case & WPS & $\deg$ & $h_{\mathrm{prim}}^{2,2}$    \\ \hline \hline
N1 & $\bP(1,1,1,1,1,1)$ & $3$ & $20$ & N10 & $\bP(1,2,6,9,9,9)$ & $18$ & $14$  \\ \hline 
N2 & $\bP(1,2,2,2,2,3)$ & $6$ & $14$ & N11 & $\bP(1,2,3,6,6,6)$ & $12$ & $13$  \\ \hline
N3 & $\bP(3,3,4,4,4,6)$ & $12$ & $2$ & N12 & $\bP(1,3,8,12,12,12)$ & $24$ & $12$   \\ \hline
N4 & $\bP(1,1,1,1,2,2)$ & $4$ & $19$ & N13 & $\bP(1,3,4,4,6,6)$ & $12$ & $10$  \\ \hline
N5 & $\bP(1,1,1,3,3,3)$ & $6$ & $19$ & N14 & $\bP(1,4,5,10,10,10)$ & $20$ & $10$  \\ \hline
N6 & $\bP(1,1,4,6,6,6)$ & $12$ & $18$ & N15 & $\bP(1,6,14,21,21,21)$ & $42$ & $10$  \\ \hline
N7 & $\bP(1,1,2,4,4,4)$ & $8$ & $17$ & N16 & $\bP(2,3,3,4,6,6)$ & $12$ & $8$  \\ \hline
N8 & $\bP(1,1,2,2,3,3)$ & $6$ & $16$ & N17 & $\bP(2,3,10,15,15,15)$ & $30$ & $8$  \\ \hline
N9 & $\bP(1,2,2,5,5,5)$ & $10$ & $14$ &  &  & &   \\ \hline
\end{tabular}
\caption{}
\label{wps-table}
\end{table}

\par Cases N1 (cubic fourfolds), N2 (studied in this paper) and N3 are new to dimension $4$. The families N4-N17 have essentially appeared in Reid's list of $95$ weighted $K3$ hypersurfaces (cf. \cite{Reid_weightedk3} or \cite{Yonemura_wtk3}) by dropping the last two weights. (N.B. Imposing the existence of a Fermat member, as we do here, reduces the list of weighted $K3$ hypersurfaces to $14$ (vs. $95$) cases; see for example \cite[Table 2.2]{Yonemura_wtk3}, the relevant cases being the first $14$ rows.)

\begin{remark} 
The degree $12$ Fermat hypersurface in $\mathbb P(2,3,3,4,4,6)$ is an example of numerical $K3$ type fourfold which is not a quasi-$K3$. A general hypersurface of degree $20$ in $\mathbb P(1,4,5,5,10,15)$ is a quasi-$K3$ but the family of fourfolds of this type does not contain a Fermat type member. In fact, there are infinitely many families of quasi-$K3$s without a Fermat member; for example, the degree $2^d$ hypersurfaces in $\bP(2^{d-2}, 2^{d-2}, 2^{d-2}, 2^{d-2}, 2^{d-1} - e, 2^{d-1}+e)$ with $d \gg 0$ and $e > 0$. 
\end{remark}

\begin{remark} 
Let $Z\subset\mathbb W$ be a quasi-smooth hypersurface appearing in Table \eqref{wps-table}.  Then, there exists a weight $w$ of $\mathbb W$ such that $Z$ intersects the general member of $|\mathcal O_{\mathbb W}(w)|$ in a quasi-smooth threefold $V$ for which $H^{3,0}(V)=0$ and $H^{2,1}(V)$ has dimension equal to $|\mathcal O_{\mathbb W}(w)|$. In the case N1, take $w=1$ to obtain a $5$-dimensional family of cubic threefolds $V$.  In the case N2, the family of degree $2$ hypersurfaces in $\mathbb W$ has dimension $4$, and $V$ has $h^{2,1}=4$ by applying \eqref{hodge-numbers} to the threefold of the same degree as $Z$ in the weighted projective space $\mathbb U$ obtained by omitting one copy of the weight $w=2$ from $\mathbb W$.  In the case N3, the family of degree $4$ hypersurfaces in $\mathbb W$ has dimension $2$ and $V$ has $h^{2,1}=2$.
\end{remark}

\begin{remark} 
More generally, if $Z$ is an almost K\"ahler $V$-manifold then $H^k(Z,\mathbb Q)$ carries a pure Hodge structure of weight $k$. Accordingly, we say that $Z$ is a \emph{numerical $K3$} if the rational cohomology of $Z$ coincides with projective space except in the middle dimension (say $\dim_{\mathbb C} Z = 2n$), which is of Hodge level $2$ with $h^{n+1,n-1}=1$.
\end{remark}

\bibliography{ref}
\end{document}